\documentclass[11pt,a4paper]{amsart}

\usepackage{amsfonts,amsmath,amssymb,amsthm}
\usepackage[utf8]{inputenc}
\usepackage{todonotes}
\usepackage{tikz}
\usepackage{pgfplots}
\usetikzlibrary{calc}
\usetikzlibrary{shadows}
\usetikzlibrary{arrows}
\usetikzlibrary{patterns}
\usetikzlibrary{matrix}
\usetikzlibrary{intersections}
\usepackage{float}

\usepackage{mathabx}

\usepackage[a4paper]{geometry}
\usepackage{url}

\newcommand{\R}{\mathbb{R}}

\newcommand{\Z}{\mathbb{Z}}

\newcommand{\map}{\Phi}
\newcommand{\ellE}{\mathbf{E}}
\newcommand{\ellK}{\mathbf{K}}
\newcommand{\ellF}{\mathbf{F}}
\newcommand{\ellPi}{\mathbf{\Pi}}
\newcommand{\ellL}{\mathbf{\Lambda}}
\renewcommand {\leq}{\leqslant}

\renewcommand {\le}{\leqslant}
\renewcommand {\ge}{\geqslant}

\newcommand{\myf}{\lambda}
\newcommand{\wt}[1]{\widetilde{#1}}
\newcommand{\ol}[1]{\overline{#1}}
\newcommand{\eps}{\epsilon}

\DeclareMathOperator{\Sec}{Sec}
\DeclareMathOperator{\vol}{vol}
\DeclareMathOperator{\dist}{dist}

\theoremstyle{plain}
\newtheorem{theorem}[equation]{Theorem}    % theorem with number
\newtheorem*{theorem*}{Theorem}
\newtheorem{lemma}[equation]{Lemma}       % lemma with number
\newtheorem{proposition}[equation]{Proposition}      % lemma with number
\newtheorem{proposition*}{Proposition}
\newtheorem{corollary}[equation]{Corollary}      % lemma with number

\theoremstyle{definition}

\newtheorem{definition}[equation]{Definition}      % lemma with number

\theoremstyle{remark}

\newtheorem*{remark*}{Remark}  % theorem without number
\newtheorem{remark}[equation]{Remark}   % theorem without number

\newtheorem*{ack}{Acknowledgements}  % theorem without number

\numberwithin{equation}{section}

\title{Solid angles and Seifert hypersurfaces}
\author{Maciej Borodzik}
\address{Institute of Mathematics, University of Warsaw, ul. Banacha 2, 02-097, Warsaw, POLAND}
\address{Institute of Mathematics, Polish Academy of Science, ul \'Sniadeckich 8, Warsaw, POLAND}
\email{mcboro@mimuw.edu.pl}

\author{Supreedee Dangskul}
\address{Center of Excellence in Mathematics and Applied
Mathematics, Department of Mathematics, Faculty of Science, Chiang Mai University, Chiang Mai 50200, Thailand}
\email{supreedee.dangskul@cmu.ac.th}

\author{Andrew Ranicki}
\address{School of Mathematics, University of Edinburgh,
James Clerk Maxwell Building, Peter Guthrie Tait Road, Edinburgh EH9 3FD,
Scotland, UK}
\email{a.ranicki@ed.ac.uk}

\date{\today}

\begin{document}

\begin{abstract}
Given a smooth closed oriented manifold $M$ of dimension $n$ embedded in $\R^{n+2}$ we study properties
of the `solid angle' function $\map\colon\R^{n+2}\setminus M\to S^1$. It turns out that a non-critical level
set of $\map$ is an explicit Seifert hypersurface for $M$.
\end{abstract}
\maketitle

\section{Introduction}\label{sec:intro}

It has been known since Seifert \cite{Seifert} that every link $L=\coprod S^1\subset\R^3$ possesses a Seifert surface, that is, a compact
oriented surface $\Sigma\subset\R^3$ such that $\partial\Sigma=L$.
Seifert gave an explicit algorithm for finding
a Seifert surface from a link diagram.

In 1969 Erle \cite{Erle} %based on \cite[Theorem 14]{MiSt} 
proved that any codimension two embedding $M^n\subset\R^{n+2}$ of a closed oriented connected manifold~$M$ has a trivial normal bundle
 and admits a Seifert hypersurface $\Sigma^{n+1} \subset \R^{n+2}$ with
 $\partial \Sigma= M \subset \R^{n+2}$; see also \cite{BNR2}. The proof of the existence of the latter fact is not constructive, it relies on the Pontryagin--Thom construction applied to any smooth map $f:{\rm cl.}(\R^{n+2} \backslash M\times D^2) \to S^1$ representing the generator
 $$1 \in [\ol{\R^{n+2}\backslash M\times D^2},S^1]=H^1(\R^{n+2}\backslash M)=H_n(M)=\Z$$
 with $\Sigma=f^{-1}(t)$ for a regular value $t \in S^1$.
To the best of our knowledge, there is no known algorithm for constructing Seifert hypersurfaces in higher dimensions.

In this paper we present a construction of a Seifert hypersurface $\Sigma \subset \R^{n+2}$ as $\Sigma=\Phi^{-1}(t)$ for a 
concrete  smooth map $\Phi:\R^{n+2}\backslash M \to \R/\Z=S^1$ obtained geometrically.
The construction is based  on three-dimensional physical intuition.
Namely, suppose $M\subset\R^3$ is a loop with constant
electric current. The scalar magnetic potential $\wt{\map}$ of $M$ at a point $x\notin M$ is the solid angle subtended by $M$, that is, the signed area
of a spherical surface bounded by the image of $M$ under the radial projection, as seen from $x$; see
\cite[Chapter III]{Maxwell} or \cite[Section 8.3]{Haus}. As the complement $\R^3\setminus M$ is not simply connected, the potential $\wt{\map}$
is defined only modulo a constant, which we normalize to be $1$. The potential induces a well-defined function $\map\colon\R^3\setminus M\to\R/\Z$.
This physical interpretation suggests that there exists an open neighborhood $N$ of $M$ such that $\map|_{N\setminus M}$ is a locally trivial
fibration. In particular, a level set $\map^{-1}(t)$ should be a (possibly disconnected) Seifert surface for $M$. In \cite[Chapter VII]{Dee}
the second author proved that this is indeed the case, although the proof is rather involved. In fact, even for a circle the exact
formula for $\map$ is complicated; it was given by Maxwell in \cite[Chapter XIV]{Maxwell} in terms of power series and also by Paxton
in \cite{Paxton}. These formulae for $\map$ for the circle show that the analytic behavior of $\map$ near $M$ is quite intricate,
although we can show that $\map$ is a locally trivial fibration in $U\setminus M$ for some small neighborhood $U$ of $M$; see Section~\ref{sec:linearbehavior}.

The construction can be generalized to higher dimensions, even though the physical interpretation seems to be a little less clear. For
any closed oriented submanifold $M^n\subset\R^{n+2}$, by the result of Erle \cite{Erle} there exists a Seifert hypersurface.
For any such hypersurface
$\Sigma$ and a point $x\notin\Sigma$ we define $\wt{\map}(x)$ to be the high-dimensional solid angle of $M$, that is, the signed area
of the image of the radial projection of $\Sigma$ to the $(n+1)$-sphere of radius $1$ and center $x$. The value of $\wt{\map}(x)$ depends on the choice
of the hypersurface $\Sigma$, but it turns out that under a suitable normalization, $\map(x):=\wt{\map}(x)\bmod 1$ is independent of the choice
of the Seifert hypersurface. Moreover, there is a formula for $\map(x)$ in terms of integrals of some concrete differential forms over $M$,
so the existence of $\Sigma$ is needed only to show that $\map$ is well-defined.

As long as $t\neq 0\in\R/\Z$, the preimage $\map^{-1}(t)$ is a bounded hypersurface in $\R^{n+2}\setminus M$. If, additionally, $t$
is a non-critical value, $\map^{-1}(t)$ is smooth. To prove that $\map^{-1}(t)$ is actually a Seifert hypersurface for $M$, we need to study
the local behavior of $\map$ near $M$. It turns out that the closure of $\map^{-1}(t)$ is smooth up to boundary.
We obtain the following result, which we can state as follows.

\begin{theorem}\label{thm:main}
Let $M\subset\R^{n+2}$ be a smooth codimension 2 embedding. Let $\map\colon\R^{n+2}\setminus M\to\R/\Z$
be the solid angle map (or the scalar magnetic potential map).
\begin{itemize}
\item On the set of points $\{x\in\R^{n+2}\setminus M\colon (0,0,\ldots,0,1)\not\in\Sec_x(M)\}$, the map $\map$ is given by
\begin{multline*}
\map(x)=\int_M\frac{1}{\|y-x\|^{n+1}} \myf\left(\frac{x_{n+2}-y_{n+2}}{\|x-y\|}\right)\cdot\\
\sum_{i=1}^{n+1} (-1)^{i+1} (y_i-x_i) dy_1\wedge\ldots\widehat{dy_i}\ldots\wedge dy_{n+1},
\end{multline*}
where $\Sec_x$ is the secant map $\Sec_x(y)=\dfrac{y-x}{\|y-x\|}$ and $\myf$ is an explicit function depending on the dimension
$n$ described in \eqref{eq:pullbackofeta}.
\item Let $t\neq 0$ be a non-critical value of $\map$. Then $\map^{-1}(t)$ is a smooth (open) hypersurface whose closure is
$\map^{-1}(t)\cup M$. The closure
of $\map^{-1}(t)$ is a possibly disconnected Seifert hypersurface for $M$, which is a topological submanifold of $\R^{n+2}$, smooth up to boundary.
\item For $t\neq 0$, the preimage $\map^{-1}(t)$ has finite $(n+1)$--dimensional volume.
\end{itemize}
\end{theorem}
%It is worth pointing out that the solid angle map $\map(x)$ is a different object from the cone angle studied, for instance
%in \cite{DKKS,CKS}, in the context of constructions of quadrisecants for knots and bounding the ropelength. For $n=1$, that is, for
%classical links, the cone angle measures the length of the spherical projection, while the solid angle measures the area. Moreover, in our
%definition of the solid angle we take the orientation into account, while the standard definition of the cone angle measures the unsigned length.
%These two facts show that the study of the cone angle and the study of the solid angle are completely different objects. The cone angle function
%does not give a fibration in a neighborhood of $M$. On the other hand, the cone angle can be effectively studied by looking at
%minimal (hyper)surfaces bounding $M$, see \cite[Section 6]{CKS}. This approach does not seem to be very effective when studying the solid angle,
%because the value of $M$ does not depend on the choice of the surface.

The structure of the paper is the following. Section~\ref{sec:intro} defines rigorously the solid angle map $\map$.
Then a formula \eqref{eq:pullbackofeta}
for $\map$ in terms of an integral of an $n$--form over $M$ is given. In Section~\ref{sec:propernessofS}
we prove that for $t\neq 0$ the inverse images of $\map^{-1}(t)\subset\R^{n+1}$ are bounded.  It is also proved that $\map$ extends
to a smooth map $S^{n+2}\setminus M\to\R/\Z$. In Section~\ref{sec:linear}
we calculate explicitly $\map$ for a linear subspace. The resulting simple formula is used later
in the proof of the local behavior of $\map$ for general $M$. In Section~\ref{sec:circle}  we derive the Maxwell--Paxton formula for $\map$ if $M$
is a circle. These explicit calculations allow us to study the local behavior of $\map$ in detail and give insight for the general case.
In Sections~\ref{sec:behavior} and~\ref{sec:behavior2} we study the local behavior of $\map$ for general $M$. This is the most technical part of the paper.
We prove Theorem~\ref{thm:main} in Section~\ref{sec:seifert}.

\begin{ack}
The present article is an extended and generalized (to arbitrary dimension) 
version of the second part of the University of Edinburgh thesis of the second author \cite{Dee} written under supervision of Andrew Ranicki. A significant part of the paper was written during a Ph.D. internship of the second author in Warsaw. He expresses his gratitude to WCMS for the financial support.

The authors would also like to thank Jae Choon Cha and Brendan Owens for stimulating discussions.
The first author was supported by the National Science Center grant 2016/22/E/ST1/00040. The second author was supported by Chiang Mai University.
\end{ack}
\section{Definition of the map $\map$}\label{sec:definitionofS}
Consider a point $x\in\R^{n+2}$ and define the map $\Sec_x\colon\R^{n+2}\setminus\{x\}\to S^{n+1}$ given by
\[\Sec_x(y)=\frac{y-x}{\|y-x\|}.\]
The map $\Sec_x$ can be defined geometrically as the radial projection from $x$ onto the sphere:
for a point $y\neq x$ take a half-line $l_{xy}$ stemming from $x$ and passing through $y$.
We define $\Sec_x(y)$ as the unique point of intersection of $l_{xy}$ and $S_x$, where $S_x$ is the unit sphere with center $x$.

Let $\omega_{n+1}$ be the $(n+1)$-form
\[\omega_{n+1}=\sum_{j=1}^{n+2}(-1)^{j+1}u_jdu_1\wedge\ldots\wedge\widehat{du_j}\wedge\ldots\wedge du_{n+2}\]
on $S^{n+1}$. Define also
\[\sigma_{n+1}=\int_{S^{n+1}} \omega_{n+1},\]
that is, the volume of the unit $(n+1)$--dimensional sphere; for instance, $\nobreak{\sigma_1=2\pi}$, $\sigma_2=4\pi$.

Let $M$ be a closed
oriented connected and smooth manifold in $\R^{n+2}$ with $\dim M=n$.
\begin{definition}
A compact oriented $(n+1)$-dimensional submanifold $\Sigma$ of $\R^{n+2}$ such that $\partial\Sigma=M$ is called
a \emph{Seifert hypersurface} for $M$.
\end{definition}
\begin{remark}
For simplicity, throughout the paper we drop the assumption that $\Sigma$ be connected.
\end{remark}
By Erle \cite{Erle} any closed oriented submanifold $M\subset\R^{n+2}$ admits a Seifert hypersurface. Given such a surface $\Sigma$,
consider $x\in \R^{n+2}\setminus\Sigma$.
The map $\Sec_x$ restricts to a map from $\Sigma$ to $S^{n+1}$, which we will still denote by $\Sec_x$.
\begin{definition}
The \emph{solid angle} of $\Sigma$ viewed from $x$ is defined as
\begin{equation}\label{eq:defofwtmap}
\wt{\map}(x)=\frac{1}{\sigma_{n+1}}\int_{\Sigma}\Sec_x^*\omega_{n+1}.
\end{equation}
In other words, $\wt{\map}(x)$ is a signed area of a spherical surface spanned by $\Sec_x(M)$, that is, the radial projection of $M$ from the
point $x$.
\end{definition}
\begin{remark}
One should not confuse the solid angle with the cone angle studied extensively by many authors, like \cite{DKKS,CKS,Denne}.
To begin with, the cone angle is unsigned and takes values in $\R_{\ge 0}$, whereas the solid angle is an element in $\R/\Z$.
This indicates that there exist fundamental differences between the two notions.
\end{remark}

We have the following fact.
\begin{lemma}\label{lem:existence}
The value $\wt{\map}(x)\bmod 1$ does not depend on the choice of $\Sigma$. In particular $\wt{\map}$ induces a well-defined function
\[\map\colon\R^{n+2}\setminus M\to\R/\Z\cong S^1.\]
\end{lemma}
\begin{proof}
Take another surface $\Sigma'$. For simplicity assume that the interiors of $\Sigma$ and $\Sigma'$ intersect transversally.
Let $\Xi=\Sigma\cup\Sigma'$. It is an exercise in Mayer-Vietoris sequence to see that $H_{n+1}(\Xi,\Z)\cong\Z^{r+r'+s-2}$,
where $r$ and $r'$ are numbers of connected components of $\Sigma$ and $\Sigma'$, while $s$ is the number of connected components
of $\Sigma\cap\Sigma'$. The union $\Sigma\cup\Sigma'$ is a cycle and it defines an element in $H_{n+1}(\Xi,\Z)$.

On the other
hand, $\omega_{n+1}$ is a generator of $H^{n+1}(S^{n+1};\Z)$. The pull-back $\frac{1}{\sigma_{n+1}}\Sec_x^*\omega_{n+1}$
belongs to $H^{n+1}(\Xi;\Z)$.
With this point of view the integral
\[\int_{\Xi}\Sec_x^*\omega_{n+1}\]
can be regarded as the evaluation of an integral cohomology class on an integral cycle of $\Xi$, so it is an integer.
Therefore,
\[\int_{\Sigma}\Sec_x^*\omega_{n+1}-\int_{\Sigma'}\Sec_x^*{\omega_{n+1}}\in\Z.\]

This shows that $\map=\wt{\map}\bmod 1$ is well defined.
To determine the domain of $\map$, notice that for any point $x\notin M$ there exists $\Sigma$
as above that misses $x$. This means that $\map$ is defined on the whole complement of $M$.
\end{proof}

From the definition of $\map$, we recover its first important property.
\begin{proposition}\label{prop:smooth}
The map $\map$ is  smooth away from the complement of $\nobreak{M\subset\R^{n+2}}$.
\end{proposition}
\begin{proof}
Take a point $y\notin M$. There exists a smooth compact surface $\Sigma$ such that $\partial\Sigma=M$ and $y\notin\Sigma$. Then, a small neighborhood $U$
of $y$ is disjoint from $\Sigma$. Thus, the map $\Sec_x$ depends smoothly on the parameter $x$. It follows that $\wt{\map}$ is smooth in $U$.
\end{proof}

\subsection{$\map$ via integrals over $M$}

The fact that the definition of $\map(x)$ involves a choice of a Seifert hypersurface $\Sigma$ is quite embarrassing. In fact, it might be hard
to find estimates for $\map$ because we have little control over $\Sigma$. We want to define $\map$ via integrals over $M$ itself.
The key tool will be the Stokes' formula.
We use the fact that while the volume form $\omega_{n+1}$ on $S^{n+1}$ itself is not exact, its restriction $\omega'$ to
the punctured sphere $S^{n+1}\setminus \{z\}$ is.

We need the following result.
\begin{proposition}\label{prop:Seifertomits}
Let $x\in\R^{n+2}\setminus M$ and let $z\in S^{n+1}$ be such that $z\notin\Sec_x(M)$. Then, there exists a Seifert hypersurface $\Sigma$ for $M$
such that $\Sec_x(\Sigma)$ misses $z$.
\end{proposition}
\begin{remark}\label{rem:znotinM}
The result is non-trivial in the sense that
one can construct a Seifert surface $\Sigma$ even for an unknot in $\R^3$ such that the restriction $\Sec_x|_\Sigma$ is onto.
However, notice that $\Sec_x|_M$ is never onto $S^{n+1}$ in general because $\dim M<\dim S^{n+1}$.
\end{remark}
\begin{proof}
Let $H$ be the half line $\{x+tz,\,t > 0\}$. In other words $H=\Sec_x^{-1}(z)$.
\begin{figure}
\begin{tikzpicture}
\begin{scope}[xshift=-2cm,scale=0.7]
\shade[inner color=green, outer color=green!10!white, draw=green!50!black]
(0,0) ++ (10:2.2cm) node[right] {$\Sigma$} arc (10:80:2.2cm) -- ++ (80:1.6cm)  arc (80:10:3.8cm) -- cycle;
\fill[rotate=45, color=red!30!white, draw=black] (0,0) ++ (0:3cm) ellipse (0.3cm and 0.4cm);
\fill[color=black] (0,0) circle (0.1cm) node[above] {$x$};
\draw (0,0) ++ (42:3cm) node [scale=0.8,below=1mm] {$D$};
\draw (0,0) -- ++ (45:1cm) node[above] {$H$} -- ++(45:1cm) ++ (45:1cm) -- ++(45:2.5cm);
\fill (0,0) ++ (45:3cm) circle (0.1cm) node[above=2mm] {$w'$};
\end{scope}
\begin{scope}[xshift=+2cm,scale=0.7]
\shade[inner color=green, outer color=green!10!white, draw=green!50!black]
(0,0) ++ (10:2.2cm) arc (10:80:2.2cm) -- ++ (80:1.6cm)  arc (80:10:3.8cm) -- cycle;
\shade[inner color=blue!50!white, outer color=blue!10!white, draw=black]
(0,0) ++ (38:4.5cm) -- ++(218:1.5cm) arc (300:150:0.4cm) -- ++ (48:1.5cm) node[below,left=1mm,scale=0.8] {$T$}
--cycle;
%\draw (0,0) ++ (38:3cm) arc (300:150:0.4cm);
\shade[inner color=orange, outer color=orange!10!white, draw=orange!50!black]
(0,0) ++ (10:4.5cm) arc (10:80:4.5cm) -- ++ (80:2cm) arc (80:10:6.5cm) node[below] {$S'$} -- cycle;
\fill[color=black] (0,0) circle (0.1cm) node[above] {$x$};
\shade[rotate=45, left color=blue!10!white, right color=blue!20!white, draw=black] (0,0) ++ (0:5.5cm) ellipse (0.3cm and 0.4cm);
\draw (0,0) ++ (48:4.5cm) node[above=3mm,scale=0.8] {$\partial_+T$};
\draw (0,0) -- ++ (45:1cm) node[above] {$H$} -- ++(45:1cm) ++ (45:3.5cm) -- ++(45:1.5cm);
%\fill (0,0) ++ (45:3cm) circle (0.1cm) node[above] {$w'$};
\end{scope}
\end{tikzpicture}
\caption{Proof of Proposition~\ref{prop:Seifertomits}. Reducing the intersection points of $H$ with $\Sigma$.
The disk $D$ is replaced by the tube $T$ and a large sphere.}\label{fig:Seifertomits}
\end{figure}
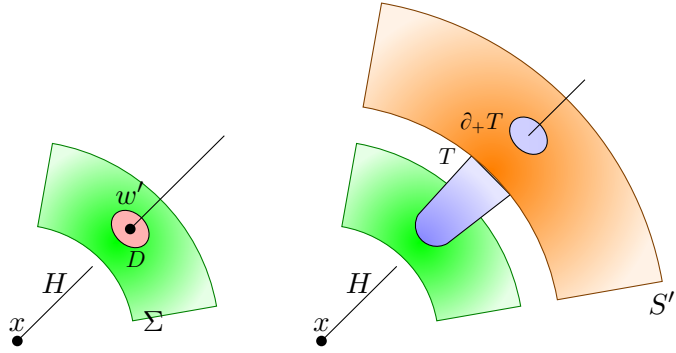
Choose any Seifert hypersurface $\Sigma$. We might assume that $H$ is transverse to $\Sigma$. The set of intersection points
of $H$ and $\Sigma$ is bounded and discrete, hence finite. Let $\{w_1,\ldots,w_m\}=H\cap\Sigma$ and assume these
points are ordered in such a way that on $H$ the point $w_1$ appears first (with smallest value of $t$), then $w_2$ and so on.

Choose the last point $w_m$ of this intersection and a small disk $D\subset\Sigma$
with center $w_m$. We can make $D$ small enough so that for any $w'\in D$ the intersection
\[\{x+tw',t>1\}\cap\Sigma\]
is empty. Set now
\[T=\{x+tw',t\ge 1,w'\in D\}\textrm{ and }\partial_v T=\{x+tw',t\ge 1,w'\in\partial D\}.\]
Consider a sphere $S=S(x,r)$, where $r$ is large. Set $S'=S\setminus (S\cap T)$.
Increasing $r$ if necessary we may and will assume that $S'$ is disjoint from $\Sigma$.
The new Seifert hypersurface
is defined as
\[\Sigma'=(\Sigma\setminus D)\cup [(\partial_v T)\cap B(x,r)]\cup S'.\]
With this construction we have $H\cap\Sigma'=\{w_1,\ldots,w_{m-1}\}$. Repeating this construction finitely many times
we obtain a Seifert hypersurface disjoint from $H$.
\end{proof}

Let $\eta_z$ be an $n$--form on $S^{n+1}\setminus\{z\}$ such that $d\eta_z=\omega_{n+1}$ and suppose $\Sigma$ is a Seifert hypersurface for $M$
such that $z\notin\Sec_x(\Sigma)$. By Stokes' formula
\[\int_{\Sigma}\Sec_x^*\omega_{n+1}=\int_{M}\Sec_x^*\eta_z.
\]
Therefore we obtain the following formula for $\map$:
\begin{equation}\label{eq:defmap}
\map(x)=\frac{1}{\sigma_{n+1}}\int_M\Sec_x^*\eta_z\bmod 1.
\end{equation}
The necessity of making the map modulo 1 comes now from different choices of the point $z\in S^{n+1}\setminus M$.

We shall need an explicit formula for $\eta_z$.
For simplicity, we consider the case when $z=(0,0,\ldots,1) \in S^{n+1}\subset \mathbb R^{n+2}$ and define $\eta:= \eta_z$; the general case
can be obtained by rotating the coordinate system.  We start with the following proposition.
\begin{proposition}
Set $z=\{0,\ldots,0,1\}$. Let $\myf: S^{n+1} \setminus \{z\} \subset\mathbb R^{n+2} \to \mathbb R$ be a smooth function with variables $u=(u_1,u_2,\ldots,u_{n+2})$. If $\myf$ involves only $u_{n+2}$ (write $\myf(u) = \myf(u_{n+2})$ for convenience) and satisfies
\begin{align} \label{diff_equa}
(1-u^2_{n+2}) \myf'(u_{n+2}) - (n+1) u_{n+2}\myf(u_{n+2}) = (-1)^n,
\end{align}
then on $S^{n+1}\setminus\{z\}$ we have
$$d\left(\myf(u_{n+2}) \omega_n\right) = \omega_{n+1}.$$
\end{proposition}
\begin{proof}
Note that
$$u_1^2 + \cdots + u_{n+2}^2 = 1 \quad \textup{implies} \quad u_1du_1 + \cdots + u_{n+2}du_{n+2} = 0,$$
and hence
$$u_i du_1\wedge\cdots\wedge du_{n+1}=(-1)^{n+i}u_{n+2}du_1\wedge\cdots\wedge\widehat{du_i}\wedge\cdots\wedge du_{n+2} $$
for all $i\in\{1,2,\ldots,n+1\}$. Therefore, by \eqref{diff_equa},
\begin{align*}
&d\left(\myf(u_{n+2}) \omega_n\right) - \omega_{n+1} =\myf'(u_{n+2})du_{n+2}\wedge\omega_n+\myf(u_{n+2})d\omega_n-\omega_{n+1}=\\
&= \frac{(-1)^n + (n+1)u_{n+2}\myf(u_{n+2}) - \myf'(u_{n+2})(1-u^2_{n+2})}{u_{n+2}}du_1\wedge\cdots\wedge du_{n+1}
\end{align*}
is a zero $(n+1)$-form.
\end{proof}
To obtain a formula for $\eta$, it remains to solve \eqref{diff_equa}. Rewriting \eqref{diff_equa}, we have
\begin{align*}
\myf'(u_{n+2}) - \frac{(n+1) u_{n+2}}{(1-u^2_{n+2})}\myf(u_{n+2}) = \frac{(-1)^n}{(1-u^2_{n+2})}.
\end{align*}
The integrating factor of this ordinary differential equation is $(1-u^2_{n+2})^{\frac{n+1}{2}}$, so the general solution of \eqref{diff_equa} can be written as
\begin{equation} \label{general_sol}
(1-u^2_{n+2})^{\frac{n+1}{2}} \myf(u_{n+2}) = (-1)^n \int (1-u^2_{n+2})^{\frac{n-1}{2}} du_{n+2}.
\end{equation}
The requirement that the solution be smooth at $u_{n+2}=-1$ translates into the following
formula
\begin{equation}\label{eq:gen_sol2}
\myf(u_{n+2})=(-1)^n{(1-u_{n+2}^2)}^{-\frac{n+1}{2}}\int_{-1}^{u_{n+2}} {(1-s^2)}^{\frac{n-1}{2}}ds.
\end{equation}
The integral in \eqref{eq:gen_sol2} can be explicitly calculated. If $n$ is odd, the result is a polynomial. If $n$ is even,
successive integration by parts eventually reduces the integral to $\int\sqrt{1-s^2}\,ds$. For small values of $n$, the function
$\myf$ is as follows.
\begin{align*}
n&=1;& \myf(u_3)&=(u_3-1)^{-1}\\
n&=2;& \myf(u_4)&=-\frac14(\pi+2u_4\sqrt{1-u_4^2}+2\arcsin u_4)(u_4^2-1)^{-3/2}\\
n&=3;& \myf(u_5)&=\frac13(u_5-2)(u_5-1)^{-1}.
\end{align*}
\begin{definition}\label{def:whatiseta}
From now on, we shall assume that $\eta=\myf(u_{n+2})\omega_n$, where $\myf$ is as in \eqref{eq:gen_sol2}.
\end{definition}
We see that $\myf$ is smooth for $u_{n+2}\in[-1,1)$ and has a pole at $u_{n+2}=1$. We shall work mostly in regions, where $u_{n+2}$ is bounded away
from $1$, so that $\myf$ and its derivatives will be bounded.

\subsection{The pull-back of the form $\eta$}\label{sec:pullback}
We shall gather some formulae for evaluating the pull-back $\Sec_x^*\eta$. This will allow us to estimate
the derivative of~$\map$.

First notice that
\begin{equation}\label{eq:diff}
d_y\frac{y_i-x_i}{\|y-x\|}=\frac{dy_i}{\|y-x\|}+(y_i-x_i)d_y\|y-x\|^{-1},
\end{equation}
where we $d_y$ means that we take the exterior derivative with respect to the $y$ variable. Consider the expression
\[\Sec_x^* du_1\wedge\cdots\wedge\widehat{du_i}\wedge\cdots\wedge du_{n+1}.\]
To calculate the pull-back we replace $du_i$ by $d_y\frac{y_i-x_i}{\|y-x\|}$. Notice that if in the wedge product the term
$(y_i-x_i)d_y\|y-x\|^{-1}$ from \eqref{eq:diff} appears twice or more, this term will be zero. Therefore
the pull-back takes the form
\begin{multline*}
\Sec_x^* du_1\wedge\cdots\wedge\widehat{du_i}\wedge\cdots\wedge du_{n+1}=\frac{1}{\|y-x\|^n}dy_1\wedge\cdots\wedge\widehat{dy_i}\wedge\cdots dy_{n+1}+\\
\frac{1}{\|y-x\|^{n-1}}\sum_{j\neq i} (-1)^{\theta(i,j)}(y_j-x_j) d_y\|y-x\|^{-1}\wedge dy_1\wedge\cdots\wedge\widehat{dy_i,dy_j}\wedge\cdots\wedge dy_{n+1},
\end{multline*}
where $\theta(i,j)$ is equal to $j-1$ if $j<i$ and $j-2$ if $j>i$.
Using the above expression together with
$$d_y\|y-x\|^{-1} = \frac{-1}{\|y-x\|^3} \left((y_1 - x_1)dy_1 + \cdots+ (y_{n+1} - x_{n+1}) dy_{n+1}\right),$$
we can calculate the pull-back of the form
$$\omega_{n}=\sum_{i=1}^{n+1} (-1)^{i+1}u_idu_1\wedge\cdots\wedge\widehat{du_i}\wedge\cdots\wedge du_{n+1}.$$
We calculate
\begin{multline*}
\Sec_x^*\omega_{n}=\frac{1}{\|y-x\|^{n+1}}\sum_{i=1}^{n+1}(-1)^{i+1}(y_i-x_i)dy_1\wedge\cdots\wedge\widehat{dy_i}\wedge\cdots\wedge dy_{n+1}+\\
+\frac{1}{\|y-x\|^{n+3}}\sum_{i=1}^{n+1}\sum_{j\neq i} (-1)^{i+\theta(i,j)} (y_i-x_i)^2(y_j-x_j)dy_i\wedge dy_1\wedge\cdots\wedge\widehat{dy_i,dy_j}\wedge\cdots\wedge dy_{n+1}\\
+\frac{1}{\|y-x\|^{n+3}}\sum_{i=1}^{n+1}\sum_{j\neq i} (-1)^{i+\theta(i,j)} (y_i-x_i)(y_j-x_j)^2 dy_j \wedge y_1\wedge\cdots\wedge\widehat{dy_i,dy_j}\wedge\cdots\wedge dy_{n+1}.
\end{multline*}
Notice that we can change the order of the sums in the last term of the above expression to be $\sum_{j=1}^{n+1}\sum_{i\neq j}$. Since $i+\theta(i,j) = j+\theta(j,i) \pm 1$, the last two sums cancel out. Hence, we obtain
\begin{equation}\label{eq:pullbackofomega}
\Sec_x^*\omega_{n}=\frac{1}{\|y-x\|^{n+1}}\sum_{i=1}^{n+1}(-1)^{i+1}(y_i-x_i)dy_1\wedge\cdots\wedge\widehat{dy_i}\wedge\cdots\wedge dy_{n+1}.
\end{equation}
In particular, using Definition~\ref{def:whatiseta} we get a proof of the first part of Theorem~\ref{thm:main}.
\begin{equation}\label{eq:pullbackofeta}\begin{split}
\Sec_x^*\eta=&\frac{1}{\|y-x\|^{n+1}}\myf\left(\frac{y_{n+2}-x_{n+2}}{\|y-x\|}\right)\\
&\cdot\sum_{i=1}^{n+1}(-1)^{i+1}(y_i-x_i)dy_1\wedge\cdots\wedge\widehat{dy_i}\wedge\cdots\wedge dy_{n+1}.
\end{split}
\end{equation}
If $n=1$ we obtain the following explicit formula, used in \cite{Dee}.
\begin{equation}\label{eq:formulan1}
\map(x_1,x_2,x_3)=\frac{1}{4\pi}\int_M \dfrac{(y_2-x_2)dy_1-(y_1-x_1)dy_2}{\|y-x\|^2(1-\dfrac{y_3-x_3}{\|y-x\|})}.
\end{equation}
It is worth mentioning the formula for $n=1$ and a general $z$ (not necessarily $(0,0,1)$), which was given in \cite[Theorem 5.3.7]{Dee}.
\[
\map(x_1,x_2,x_3)=\frac{1}{4\pi}\int_M\frac{\left(\dfrac{y-x}{\|y-x\|}\times z\right)\cdot Dy}{\|y-x\|\left(1-\dfrac{y-x}{\|y-x\|}\cdot z\right)},
\]
where $Dy=(dy_1,dy_2,dy_3)$.

We conclude by remarking that if
$$\omega_{n+1}=\sum_{i=1}^{n+2} (-1)^{i+1} u_i du_1\wedge\cdots\wedge\widehat{du_i}\wedge\cdots\wedge du_{n+2},$$
then analogous arguments as those that led to formula~\eqref{eq:pullbackofomega} imply that
\begin{equation}\label{eq:secondpullback}
\Sec_x^*\omega_{n+1}=\frac{1}{\|y-x\|^{n+2}}\sum_{i=1}^{n+2} (-1)^{i+1} (y_i-x_i)dy_1\wedge\cdots\wedge\widehat{dy_i}\wedge\cdots\wedge dy_{n+2}.
\end{equation}

\subsection{Estimates for derivatives of $\Sec_x^*\eta$}
% write more about \eta_z.
The following results are direct consequences of the pull-back formula for $\eta$, \eqref{eq:pullbackofeta}. We record them
for future use in Sections~\ref{sec:propernessofS} and~\ref{sec:behavior}. Recall from Section~\ref{sec:pullback} that $\eta$ was defined as a form on
$S^{n+1}\setminus(0,\ldots,0,1)$. The form $\eta_z$ for general $z\in S^{n+1}$ is obtained by rotation of the coordinate system.
\begin{lemma}\label{lem:technicalestimate}
For any $m\ge 0$, there exists a constant $C^{\#}_{m,n}$ such that for each non-negative integers $k_1,\ldots,k_{n+2}$ such
that $\sum k_i=m$, the (higher) differential of the pull-back $\Sec_x^*\eta$ has the form
\[\frac{\partial^m}{\partial x_1^{k_1}\cdots \partial x_{n+2}^{k_{n+2}}}\Sec_x^*\eta=\sum_{i=1}^{n+1}H_idy_1\wedge\cdots\wedge\widehat{dy_i}\wedge\cdots
\wedge dy_{n+1},\]
where
\begin{equation}\label{eq:Hi}
H_i=\sum_{j=0}^m \myf^{(j)}\left(\frac{y_{n+2}-x_{n+2}}{\|y-x\|}\right) H_{i}^j,
\end{equation}
and $H_i^j$ are smooth functions satisfying $|H_i^j|\le C^{\#}_{m,n}\|y-x\|^{-(n+m-j)}$.
\end{lemma}
\begin{proof}
If $m=0$, the proof is a direct consequence of \eqref{eq:pullbackofeta}. The general case follows by an easy induction.
\end{proof}
As a consequence of Lemma~\ref{lem:technicalestimate} we prove the following fact.
\begin{lemma}\label{lem:mainbound}
For any $D<1$ and for any integer $m>0$, there is a constant $C^{D}_{n,m}$ such that if $z\in S^{n+1}$,
$y,x$ satisfy $\langle \frac{y-x}{\|y-x\|},z\rangle<D$
and $\sum k_i=m$, then
the derivative $\dfrac{\partial^m}{\partial x_1^{k_1}\cdots \partial x_{n+2}^{k_{n+2}}}\Sec_x^*\eta_z$ is a sum
of forms of type $H_{i_1,\ldots, i_n} dy_{i_1}\wedge\ldots\wedge dy_{i_n}$, where all the coefficients $H_{i_1,\ldots,i_n}$ are bounded by
$C^D_{n,m}\|y-x\|^{-n-m}$.
\end{lemma}
\begin{proof}
Apply a linear orthogonal map of $\R^{n+2}$ that takes $z$ to $(0,0,\ldots,0,1)$. Let $x'$ and $y'$ be the images of $x$ and $y$,
respectively, under this map. We have $\|y'-x'\|=\|y-x\|$ and the condition
The condition $\langle\frac{y-x}{\|y-x\|},z\rangle<D$ becomes $\frac{y_{n+2}'-x_{n+2}'}{\|y-x\|}<D$.
We will use \eqref{eq:Hi}.  As $D<1$, on the interval $[-1,D]$ the function $\myf$ and its derivatives up to $m$-th inclusive are
bounded above by some constant $C_{D,m}$ depending on $D$ and $m$.
The constant $C^D_{n,m}$ can be chosen as $C^D_{n,m}=(m+1)C^{\lambda}_{D,m}C^{\#}_{n,m}$.
\end{proof}

\section{Properness of $\map$}\label{sec:propernessofS}

\begin{theorem}\label{thm:mapisproper}
For any $t\in(0,1)$ there exists $R_t$ such that $\map^{-1}(t)\subset B(0,R_t)$. In other words, all fibers of $\map$ except $\map^{-1}(0)$
are bounded.
\end{theorem}
\begin{proof}
Choose a Seifert hypersurface $\Sigma$ for $M$. We may assume that it is contained in a ball $B(0,r)$ for some $r>0$. As $\Sigma$ is compact
and smooth, there exists a constant $C_\Sigma$ such that if an $(n+1)$--form $\omega_{n+1}$ on $\R^{n+2}$ has all the
coefficients bounded from above by $T$, then
$|\int_\Sigma\omega_{n+1}|<C_\Sigma T$.

Now take $R\gg 0$ and suppose $x\notin B(0,R+r)$. Then the distance of $x$ to any point $y\in\Sigma$ is at least $R$. Then $\Sec_x^*\omega_{n+1}$
has all the coefficients bounded by $R^{-n-1}$, see \eqref{eq:secondpullback}, and
therefore $|\int_\Sigma\Sec_x^*\omega_{n+1}|\le C_\Sigma R^{-1-n}$. This means that
\[\map(\R^{n+2}\setminus B(0,R+r))\subset(-C_\Sigma R^{-1-n},C_\Sigma R^{-1-n}),\]
or equivalently, that if $t\notin (-C_\Sigma R^{-1-n},C_\Sigma R^{-1-n})$, then $\nobreak{\map^{-1}(t)\subset B(0,R+r)}$.
\end{proof}

\begin{corollary}\label{cor:extendstosphere}
The map $\map$ extends to a $C^{n+1}$ smooth map from $S^{n+2}\setminus M$ to $S^1$.
\end{corollary}
\begin{proof}[Sketch of proof]
Smoothness of $\map$ at infinity is equivalent to the smoothness of $w\mapsto \map(\frac{w}{\|w\|^2})$ at $w=0$. The proof of
Theorem~\ref{thm:mapisproper} generalizes to show that for any $m>0$ there exists $C_m$ with a property that
$|D^\alpha\map(x)|\le C_m\cdot\|x\|^{-n-1-|\alpha|}$, and $\|D^\alpha\frac{w}{\|w\|^2}\|\le C_m\|w\|^{-|\alpha|-1}$ whenever $|\alpha|\le m$. Here
$\alpha$ is a multi-index.

Now by the di Bruno's formula for higher derivatives of the composite function, we infer that
$|D^\alpha \map(\frac{w}{\|w\|^2})|\le C\|w\|^{n+2-|\alpha|}$ (the worst case occurs when $\map$ is differentiated only once, while
$\dfrac{w}{\|w\|^2}$ is differentiated $|\alpha|$ times). Hence, the limit at $w\to 0$ of all derivatives of $w\mapsto \map(\dfrac{w}{\|w\|^2})$ of order up to
$n+1$ is zero.
\end{proof}

We can also strengthen the argument of Theorem~\ref{thm:mapisproper} to obtain a more detailed information about the behavior of $\map$ at a large scale.
\begin{theorem}\label{thm:euler}
Suppose $\Sigma$ is a Seifert hypersurface and $r$ is such that $\nobreak{\Sigma\subset B(0,r)}$. For any $R>r$, if $\|x\|>R$ we have
\[\left\vert\sum_{i=1}^{n+2}x_i\frac{\partial\map}{\partial x_i}+(n+1)\map\right\vert\le C_\Sigma(n+2)\frac{rR^{n+2}}{(R-r)^{n+2}}\|x\|^{-(n+2)},\]
where $C_\Sigma$ depends solely on $\Sigma$ and not on $R$ and $r$.
\end{theorem}
\begin{proof}
Using \eqref{eq:secondpullback} write
\begin{multline*}
\frac{\partial\Sec_x^*\omega_{n+1}}{\partial x_i}=\frac{-1}{\|y-x\|^{n+2}} (-1)^{i+1} dy_1\wedge\ldots\wedge\widehat{dy_i}\wedge\ldots\wedge dy_{n+2}+\\
+(n+2)\frac{y_i-x_i}{\|y-x\|^{n+4}}\sum_{j=1}^{n+2} (-1)^{j+1}(y_j-x_j)dy_1\wedge\ldots\wedge\widehat{dy_j}\wedge\ldots\wedge dy_{n+2}.
\end{multline*}
This implies that
\[\sum_{i=1}^{n+2}x_i\frac{\partial\Sec_x^*\omega_{n+1}}{\partial x_i}=-\xi_1+(n+2)\xi_2,\]
where
\begin{align*}
\xi_1&=\frac{1}{\|y-x\|^{n+2}}\sum_{i=1}^{n+2} (-1)^{i+1}x_idy_1\wedge\ldots\wedge\widehat{dy_i}\wedge\ldots\wedge dy_{n+2}\\
\intertext{and}
\xi_2&=\frac{1}{\|y-x\|^{n+4}}\sum_{i=1}^{n+2}\sum_{j=1}^{n+2} (-1)^{j+1} (y_i-x_i)(y_j-x_j)x_i dy_1\wedge\ldots\wedge\widehat{dy_j}\wedge\ldots\wedge dy_{n+2}.\\
\intertext{Write also}
\xi_3&=\frac{1}{\|y-x\|^{n+2}}\sum_{i=1}^{n+2}(-1)^{i+1} (y_i-x_i)dy_1\wedge\ldots\wedge\widehat{dy_i}\wedge\ldots\wedge dy_{n+2}.
\end{align*}
Now suppose $\|x\|>R$ and $\|y\|<r$.
Then $-\xi_1-\xi_3$ has all the coefficients bounded by $\frac{rR^{n+2}}{(R-r)^{n+2}}\|x\|^{-n-2}$. Likewise notice that
\[\left|\sum_{i=1}^{n+2}(y_i-x_i)x_i+\|y-x\|^2\right|=\left|\sum_{i=1}^{n+2}(y_i-x_i)y_i\right|\le \|y\| \|y-x\|,\]
where we used Schwarz' inequality in the last estimate.
Therefore $-\xi_2-\xi_3$ has all the coefficients bounded by $\frac{\|y\|}{\|y-x\|^{n+2}}$, and 
by assumptions on $\|x\|$ and $\|y\|$ we have that
\[\frac{\|y\|}{\|y-x\|^{n+2}}\le\frac{rR^{n+2}}{(R-r)^{n+2}}\|x\|^{-n-2}.\]
We conclude that
\begin{equation}\label{eq:euler2}
\left|\int_\Sigma -\xi_1+(n+2)\xi_2+(n+1)\xi_3\right|\le C_\Sigma(n+2)\frac{rR^{n+2}}{(R-r)^{n+2}} \|x\|^{-(n+2)}.
\end{equation}
As $\map(x)=\int_\Sigma \xi_3$, we obtain the statement.
\end{proof}
The statement of Theorem~\ref{thm:euler}, in theory, can be used to obtain information about $C_\Sigma$ from the behavior of $\map$
at infinity. The left hand side of \eqref{eq:euler2} is equal to $\left|\sum\limits_{i=1}^{n+2} x_i\frac{\partial\map}{\partial x_i}+(n+1)\map\right|$,
and does not depend on $\Sigma$. Therefore if we know $\map$ and its derivatives, we can find a lower bound for $C_\Sigma$, which
roughly tells, how complicated $\Sigma$ might be. Unfortunately we do not know of any examples where this can be used effectively.

\section{$\map$ for an $n$-dimensional linear surface}\label{sec:linear}
Let $M\subset\R^{n+2}$ be given by $\{w\in\R^{n+2}\colon w_1=0,w_2=0\}$, the set of points having the first two coordinates zero. We wish to calculate the map $\map$ for $M$. We encounter some technical problems.
Firstly, as $M$ is not compact, we have no reason to expect that $\map$ has bounded fibers and indeed, the statement of Theorem~\ref{thm:mapisproper}
does not hold. Secondly, there is a more serious problem. The map $\map$ will depend on the choice of the ``Seifert hypersurface''. We used
quotation marks in the previous sentence because $M$, as it is not compact, does not admit a compact Seifert hypersurface. However, if we choose a
Seifert hypersurface for $M$ to be an $(n+1)$--dimensional half-space, it turns
out that the derivative of $\map$ does not depend on the choice of the half-space. This feature and calculations for
$\frac{\partial\map}{\partial x_j}$ will be important in Section~\ref{sec:behavior}.

Set $\Sigma=\{w:w_1\le 0,\ w_2=0\} \subset\mathbb R^{n+2}$.
For any point $x\notin\Sigma$, the value of the map $\map(x)$
is (up to a sign) the area of the image $\Sec_x(\Sigma)$. This image can be calculated explicitly.

Choose a point $y=(y_1,\ldots,y_{n+2})\in S^{n+1}$. The half-line from $x\notin\Sigma$ through $x+y$ is given by $t\mapsto x+ty$, $t\ge 0$; see
Figure~\ref{fig:halfline}.
By definition, $y\in\Sec_x(\Sigma)$ if and only if this half-line intersects $\Sigma$, that is, for some $t_0>0$
we have
\begin{equation}\label{eq:newdisplayed}
x_2+t_0y_2=0\textrm{ and }x_1+t_0y_1\le 0.
\end{equation}
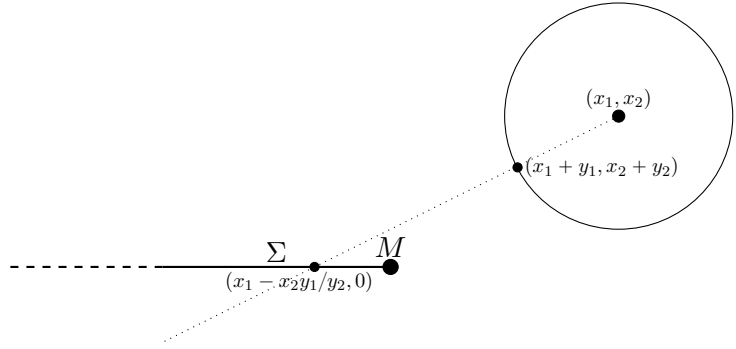
\begin{figure}
\begin{tikzpicture}
\draw[thick, dashed](-5,0) -- (-3,0);
\draw[thick] (-3,0) -- (0,0);
\draw[fill=black] (0,0) circle (0.1cm)  node[above] {$M$};
\draw (-1.5,0.2) node {$\Sigma$};
\draw (3,2) circle (1.5cm);
\draw[fill=black] (3,2) circle (0.08cm) node [above, scale=0.7] {$(x_1,x_2)$};
\draw[dotted] (3,2) -- (-3,-1);
\draw[fill=black] (-1,0) circle (0.06cm);
\draw (-1.2,-0.2)  node [ scale=0.7] {$(x_1-x_2y_1/y_2,0)$};
\draw[fill=black] (1.67,1.32) circle (0.06cm) node [right,scale=0.7] {$(x_1+y_1,x_2+y_2)$};
\end{tikzpicture}
\caption{The half-line from $(x_1,x_2)$ through $(x_1+y_1,x_2+y_2)$ hitting the Seifert hypersurface $\Sigma$.}\label{fig:halfline}
\end{figure}
Note that if $x_2= 0$, then the half-line through $x$ and any point in $\Sigma$ will meet $M$, 
which results in an $n$-dimensional image $\Sec_x(\Sigma)$ in $S^{n+1}$.
Suppose $x_2\neq 0$.
The condition $t_0>0$ together with \eqref{eq:newdisplayed} implies that the signs of $x_2$ and $y_2$ must be opposite.
Plugging $t_0$ from the first equation of \eqref{eq:newdisplayed} into the second one, we obtain
\begin{equation}\label{eq:mainestimate}
x_1-\frac{x_2y_1}{y_2}\le 0.
\end{equation}
The calculation of $\map$ boils down to the study of the set of $x_1,x_2$ satisfying \eqref{eq:mainestimate}.
Write $x_1=r\cos2\pi\beta$ and $x_2=r\sin2\pi\beta$.
Multiply \eqref{eq:mainestimate} by $\frac{y_2}{x_2}$ (which is negative) to obtain the inequality
\[y_1\le y_2/\tan 2\pi\beta.\]
 There are four cases depending on in which quadrant of the plane contains $(x_1,x_2)$, see Figure~\ref{fig:fourcases}.
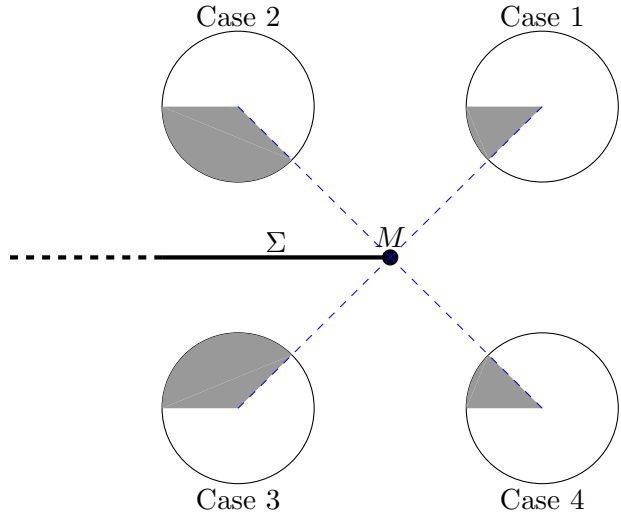
\begin{figure}
\begin{tikzpicture}
\draw[ultra thick, dashed](-5,0) -- (-3,0);
\draw[ultra thick] (-3,0) -- (0,0);
\draw[fill=black] (0,0) circle (0.1cm)  node[above] {$M$};
\draw (-1.5,0.2) node {$\Sigma$};
\draw (2,2) circle (1cm);
\draw[fill=black!40,draw=none] (1,2) arc [start angle=180, end angle=225, radius=1cm];
\draw[fill=black!40,draw=none] (2,2) -- ++(225:1cm) -- (1,2) -- cycle;
\draw[blue, dashed] (2,2) -- (0,0);
\draw(2,3.2) node {Case 1};
\draw (-2,2) circle (1cm);
\draw[fill=black!40,draw=none] (-3,2) arc [start angle=180, end angle=315, radius=1cm];
\draw[fill=black!40,draw=none] (-2,2) -- ++ (315:1cm) -- (-3,2) -- cycle;
\draw[blue, dashed] (-2,2) -- (0,0);
\draw(-2,3.2) node {Case 2};
\draw (-2,-2) circle (1cm);
\draw[fill=black!40,draw=none] (-3,-2) arc[start angle=180, end angle=45, radius=1cm];
\draw[fill=black!40,draw=none] (-2,-2) -- ++ (45:1cm) -- (-3,-2) -- cycle;
\draw[blue, dashed] (-2,-2) -- (0,0);
\draw(-2,-3.2) node {Case 3};
\draw (2,-2) circle (1cm);
\draw[fill=black!40,draw=none] (1,-2) arc [start angle=180, end angle=135, radius=1cm];
\draw[fill=black!40,draw=none] (2,-2) -- ++ (135:1cm) -- (1,-2) -- cycle;
\draw[blue, dashed] (2,-2) -- (0,0);
\draw(2,-3.2) node {Case 4};
\end{tikzpicture}
\caption{The four cases of possible position of points $(x_1,x_2)$ and the image of the projection. Instead of drawing the preimage
in the boundary on a circle, we draw a circular sector in a disk for better readability.}\label{fig:fourcases}
\end{figure}
We next calculate the area of the image $\Sec_x(\Sigma)$ of each of those four cases.
To do so, we first deal with the calculations and then discuss the choice of the sign. For the moment, we choose a sign for the area as
$\epsilon\in\{-1,+1\}$; refer to Section~\ref{sec:sign} for the discussion of the sign convention.

Notice that the area of the two--dimensional
circular sector in Figure~\ref{fig:fourcases} is (up to normalization) equal to the $(n+1)$--dimensional
area of the image $\Sec_x(\Sigma)$. This is because the defining equations are homogeneous, and other variables $y_3,\ldots,y_{n+2}$
do not enter in the definition of the region. %The calculation of $\map$ is split into four cases depending on the signs of $x_1$ and $x_2$.

%Let $2\pi\beta$ be the angle between $x=(x_1,x_2,\ldots,x_{n+2})$ and the horizontal axis, see Figure~\ref{fig:fourcases}, measured counterclockwise. If $x_2 = 0$, then $\map(x)=0$. Other cases are as follows:
%\todo[inline,color=green!10]{[Dee] The area of a circular sector is a half the angle. So, I've divided all the areas by 2.}
\begin{itemize}
\item[\textbf{Case 1}] \underline{$x_1\ge 0$ and $x_2>0$.} The region $\Sec_x(\Sigma)$ is given by $y_2 < 0$ (because the sign of $y_2$ is opposite to the sign of $x_2$),
$y_1\le y_2/\tan2\pi\beta$ and $\tan2\pi\beta\in(0,\infty)$.
The area of the sector corresponding to Case~1 is equal to $\pi\beta$, hence $\map(x)=\epsilon\beta$, where $\epsilon$ is a sign.
\item[\textbf{Case 2}] \underline{$x_1\le 0$ and $x_2 > 0$.} The region $\Sec_x(\Sigma)$ is given by $y_2 < 0$, $y_1\le y_2/\tan2\pi\beta$, where
$\tan2\pi\beta\in(-\infty,0)$. The area of the sector is equal to $\pi\beta$ and so $\map(x)=\epsilon\beta$.
\item[\textbf{Case 3}] \underline{$x_1\le 0$ and $x_2 < 0$.} Then $y_2 > 0$ and $\tan2\pi\beta\in(0,\infty)$. The area of the sector is $\pi-\pi\beta$,
but now the hypersurface $\Sigma$ is seen from the other side, hence the signed area is $\epsilon(\pi\beta-\pi)$. After normalizing and
taking modulo $1$, we obtain that $\map(x)=\epsilon\beta$.
% we take modulo 1, because this is the definition of \map.
\item[\textbf{Case 4}] \underline{$x_1\ge 0$, $x_2 < 0$.} Then $y_2 > 0$ and $\tan2\pi\beta\in(-\infty,0)$. As in Case 3, we deduce that
the area is $\pi-\pi\beta$ and we obtain $\map(x)=\epsilon\beta$.
\end{itemize}

Putting all the cases together, we see that $\map(x)=\epsilon\beta$.

Suppose we take another `Seifert surface' for $M$, denoted $\Sigma'$, given by $u_1=0$, $u_2\le 0$.
Let $\map'$ be the map $\map$ defined relatively to $\Sigma'$. To calculate $\map'$,
we could repeat the above procedure, yet we present a quicker argument. A counterclockwise rotation $A$ in the $(u_1,u_2)$-plane by angle $\frac{\pi}{2}$
fixes $M$ and takes $\Sigma$ to $\Sigma'$.
In particular $\map'(x)=\map(Ax)$. Hence $\map'(x)=\epsilon(\beta-\frac14)$
We notice that $\map'\neq\map$, but on the other hand $\map'-\map$ is a constant. This approach shows that
if we take a linear hypersurface (a half-space) for the `Seifert surface' of $\map$, then it is well defined up to a constant, and so the
derivatives are well defined.

\subsection{The sign convention}\label{sec:sign}
Given that $\map$ is defined as an integral of a differential form, changing the orientation of $M$ induces a reversal of the sign of $\map$.
We use the example of a linear surface to show how the sign is computed.

Choose an orientation of $M$ in such a way that
$\frac{\partial}{\partial u_3},\ldots,\frac{\partial}{\partial u_{n+2}}$ is a positive basis of $TM$. Stokes' theorem
is applicable if $\Sigma$ is oriented by the rule
``normal outwards first'', see \cite[Chapter 5]{Spivak},
so that $\frac{\partial}{\partial u_1},\frac{\partial}{\partial u_3},\ldots,\frac{\partial}{\partial u_{n+2}}$ is an oriented basis of $T\Sigma$.

The way of seeing the sign is by calculating $\int_\Sigma\Sec_x^*\omega_{n+1}$.
By \eqref{eq:secondpullback} we know that
\[\wt{\map}(x_1,\ldots,x_{n+2})=\int_\Sigma\Sec_x^*\omega_{n+1}=\int_{\substack{y_1\le 0\\ y_2=0}}\frac{x_2}{\|y-x\|^{n+2}} dy_1\wedge dy_3\wedge\ldots \wedge dy_{n+2}.\]
Given the orientation of $\Sigma$ we have
\begin{equation}\label{eq:check_sign}
\begin{split}
\int_{y_1\le 0,y_2=0}&\frac{x_2}{\|y-x\|^{n+2}} dy_1\wedge dy_3\wedge\ldots \wedge dy_{n+2}=\\
=&x_2\int_{-\infty}^0\left(\int\frac{1}{\|y-x\|^{n+2}} dy_3\ldots dy_{n+2}\right)dy_1.
\end{split}
\end{equation}
Notice that on the left hand side we have an integral of a differential form, whereas on the right hand side the integral is with respect
to the $(n+1)$--dimensional Lebesgue measure on a subset of $\R^{n+1}$.

The function $\int\frac{1}{\|y-x\|^2} dy_3\ldots dy_{n+2}$ is positive, therefore $\wt{\map}$ is positive for $x_2>0$,
negative for $x_2<0$ and $0$ for $x_2=0,x_1>0$ (notice that \eqref{eq:check_sign} is not defined if $x_2=0$ and $x_1\le 0$: if this holds,
the point $(x_1,x_2,\ldots,x_{n+2})$ lies on $\Sigma$ and the integral diverges). Therefore $\frac{\partial}{\partial x_2}\wt{\map}|_{x_2=0,x_1>0}$
is non-negative. This is possible only if the choice of sign is $\epsilon=+1$.

\section{$\map$ for a circle}\label{sec:circle}
We now use the formula for $\map$ via the integrals of the pull-back of $\eta$, see \eqref{eq:pullbackofeta}, to give an explicit formula for $\map$ in the
case when $M$ is a circle. The output is given in terms of elliptic integrals. Detailed calculations can be found e.g. in \cite{Dee},
therefore we omit some tedious computations. We focus on the analysis of the behavior of $\map$ near the circle.

%%%%%%%%%%%%%%%%%%%%%%%%%%%%%%%%%%%%%%%%%%%%%%%%%%%%%%%%%%%%%%%%%%%%%%%%%%%%%%%%%%%%%%%%%%%%%%%%%%%%%%%%%%%%%%%%%%
\subsection{Elliptic Integrals}
For the reader's convenience we give a quick review of elliptic integrals and their properties. We shall use these definitions in future calculations.
This section is based on \cite{Byrd}.

\begin{definition}
Let $\varphi\in [0,\pi/2]$. For any $k\in [0,1]$, the \emph{complementary modulus} $k'$ of $k$ is defined by $k'=\sqrt{1-k^2}$.
\begin{enumerate}
\item The integral
\begin{equation}\label{eq:defofF}
\ellF(\varphi,k) = \int_0^{\varphi} \dfrac{dt}{\sqrt{1-k^2\sin^2 t}}
\end{equation}
is called an \emph{elliptic integral of the first kind}. If $\varphi = \pi/2$, it is called a \emph{complete elliptic integral of the first kind}, denoted by $\ellK(k):= \ellF(\pi/2,k)$.

\item The integral
\[\ellE(\varphi,k) = \int_0^{\varphi} \sqrt{1-k^2\sin^2 t}~dt\]
is called an \emph{elliptic integral of the second kind}. If $\varphi = \pi/2$, it is called a \emph{complete elliptic integral of the second kind}, denoted by $\nobreak{\ellE(k):= \ellE(\pi/2,k)}$.
\item The integral
\[\ellPi(\varphi,\alpha^2,k) = \int_0^{\varphi} \dfrac{dt}{(1-\alpha^2\sin^2 t)\sqrt{1-k^2\sin^2 t}}\]
is called an \emph{elliptic integral of the third kind}. If $\varphi = \pi/2$, it is called a \emph{complete elliptic integral of the third kind}, denoted by $\ellPi(\alpha^2,k):= \ellPi(\pi/2,\alpha^2,k)$.
\item Heuman's Lambda function $\ellL_0 (\beta,k)$ can be defined by the formula
\[\ellL_0 (\beta,k) = \dfrac{2}{\pi}\left(\ellE(k)\ellF(\beta,k') + \ellK(k)\ellE(\beta,k') - \ellK(k)\ellF(\beta,k')\right).\]
\end{enumerate}
\end{definition}

\smallskip
%\begin{remark}
%~~
%\begin{itemize}
%\item The integrals $\ellF(\varphi,k)$ and $\ellPi(\varphi,\alpha^2,k)$ may not be defined for some values of $\varphi$.
%For example, if $\varphi = \pi/2$ and $k=1$, then the integral in \eqref{eq:defofF} is divergent.
%
%\item Some special values of elliptic integrals and the Heuman's Lambda function are
%\begin{align*}
%\ellE(0,k) = \ellF(0,k) = \ellPi(0,\alpha^2,k) = 0\\
%\ellE(\varphi,0) = \ellF(\varphi,0) = \ellPi(\varphi,0,0) = \varphi\\
% \ellK(0) = \ellE(0) = \pi/2,\quad \ellE(1) = 1\\
%\ellL_0(\beta,0) = \sin\beta, \quad \ellL_0(0,k) = 0\\
%\ellL_0(\beta,1) = 2\beta/\pi,\quad \ellL_0(\pi/2,k) = 1\\
%\ellL_0(-\beta,k) = - \ellL_0(\beta,k).
%\end{align*}
%\end{itemize}
%\end{remark}

Although $ \ellK(k)$ blows up at $k=1$,
we know how fast it goes to infinity as $k$ approaches 1 from below.

\begin{proposition}[see \expandafter{\cite[formula (10) on page 318]{Highertrans}}]\label{K(k)_growth} We have
$$ \ellK(k) = \ln \dfrac{4}{\sqrt{1-k^2}} + O\left((1-k^2)\ln\sqrt{1-k^2}\right) \quad \textup{as} \quad k\to 1^-.$$
In particular
\begin{equation}\label{eq:limKk}
\lim_{k \to 1^-} \left( \ellK(k) - \ln \dfrac{4}{\sqrt{1-k^2}}\right) = 0.
\end{equation}
\end{proposition}

The differentials of $\ellK(k)$ and $\ellE(k)$ are calculated e.g. in \cite[page 282]{Byrd}.
\begin{align}  \label{dK(k)/dk}
\dfrac{d}{dk}\ellK(k) = \dfrac{\ellE(k) - (k')^2 \ellK(k)}{k(k')^2}
\end{align}
and
\begin{align} \label{dE/dk}
\dfrac{d}{dk}\ellE(k) = \dfrac{\ellE(k) - \ellK(k)}{k}
\end{align}
where $k' = \sqrt{1-k^2}$.
The derivative of
the Heuman's Lambda function $\ellL_0(\beta,k)$ is given by the following formula, see \cite[formulae 710.11 and 730.04]{Byrd}.
\begin{align} \label{dHeuman/dk}
\dfrac{\partial}{\partial k}\ellL_0(\beta,k)& = \dfrac{2(\ellE(k)-\ellK(k))\sin\beta\cos\beta}{\pi k\sqrt{1-k'^2\sin^2 \beta}}\\
\intertext{and}
\label{dHeuman/d(angle)}
\dfrac{\partial}{\partial \beta}\ellL_0(\beta,k)& = \dfrac{2(\ellE(k)-k'^2\sin^2 \beta\, \ellK(k))}{\pi \sqrt{1-k'^2\sin^2 \beta}}.
\end{align}
%%%%%%%%%%%%%%%%%%%%%%%%%%%%%%%%%%%%%%%%%%%%%%%%%%%%%%%%%%%%%%%%%%%%%%%%%%%%%%%%%%%%%%%%%%%%%%%%%%%%%%%%%%%%%%%%%%
\subsection{Computation of $\map$ for the circle}\label{sec:computations}
In this section, we follow closely \cite[Sections 6.2 and 6.3]{Dee}.
The circle $U$ has the parametrization $\gamma:[-\pi,\pi] \to \mathbb R^3$ given by
\[ \gamma(t) = (\cos t , \sin t, 0). \]

Suppose $x\in\R^3$ is such that $x\notin\{u_1^2+u_2^2=1,\ u_3\le 0\}$. Then $\Sec_x(U)$ does not contain $(0,0,1)$ and \eqref{eq:formulan1}
implies:
\begin{equation} \label{formula_unknot}
\map(x_1,x_2,x_3)=\frac{1}{4\pi}\int_{-\pi}^{\pi} \frac{(x_1\cos t + x_2 \sin t - 1)dt}{Q+x_3\sqrt{Q}},
\end{equation}
where
\[Q=1+\|x\|^2 - 2x_1\cos t - 2x_2\sin t.\]
Write $x_1=r\cos\theta$, $x_2=r\sin\theta$ for $r\ge 0$. Substituting this into \eqref{formula_unknot}, we observe that $\map$
does not depend on $\theta$, hence we can write $\map=\map(r,x_3)$, that is,
\begin{equation} \label{formula_unknot-u_2=0}
\map(r,x_3) = \frac{1}{4\pi}\int_{-\pi}^{\pi} \dfrac{(r\cos t - 1)dt}{1+r^2+x_3^2 - 2r\cos t + x_3\sqrt{1+r^2+x_3^2 - 2r\cos t}}.
\end{equation}
We have some special cases where we can compute the integral explicitly.
If $x_3 = 0$, we use the identity
$$\cos t = \dfrac{1-\tan^2 (t/2)}{1+\tan^2(t/2)}$$
and deal with improper integrals; there are two situations:

\begin{itemize}

\item $r <1$: we have

\[\map(r,0) = \frac{1}{4\pi}\left[-\dfrac{t}{2} - \arctan \left(\dfrac{1+r}{1-r} \tan\dfrac{t}{2}\right)\right]_{-\pi}^{\pi} = -\frac12;\]

\item $r >1$: we have

\[\map(r,0) = \frac{1}{4\pi}\left[-\dfrac{t}{2} + \arctan \left(\dfrac{r+1}{r-1} \tan\dfrac{t}{2}\right)\right]_{-\pi}^{\pi} = 0.\]

\end{itemize}
This agrees with the geometric interpretation. If we choose the disk $D=\{r\le 1,\ x_3=0\}$ as a Seifert surface
for $U$, then for $x=(r\cos\theta,r\sin\theta,0)$ with $r>1$, the image $\Sec_x(D)$ is one-dimensional, so $\Phi(x)=0$.
Conversely, for $x=(r\cos\theta,r\sin\theta,0)$ with $r<1$ we choose a Seifert surface $\Sigma$
to be the disk $D$ with a smaller disk centered at $x$
replaced by a hemisphere with center at $x$. In this way, the image $\Sec_x(\Sigma)$ is a hemisphere; see Figure~\ref{fig:hemisphere}.
\begin{remark}\label{rem:tzero}
The inverse image $\map^{-1}(0)$ contains (and actually it is equal) to the set $\{x_1^2+x_2^2>1,\ x_3=0\}$. This shows
that the assumption that $t\neq 0$ in Theorem~\ref{thm:mapisproper} is necessary.
\end{remark}
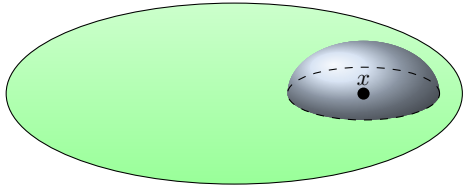
\begin{figure}
\begin{tikzpicture}
\shade[bottom color=green!40, top color=green!20, draw=black] (0,0) ellipse (3cm and 1.2 cm);
\shade[ball color=green!30!blue!20,draw=none] (0.7,0) arc (180:360:1 cm and 0.35 cm) arc (0:180:1cm and 0.7 cm);
\draw[dashed, thin] (1.7,0) ellipse (1cm and 0.35 cm);
\fill[color=black] (1.7,0) circle (0.08cm) node[above,scale=0.8] {$x$};
\end{tikzpicture}
\caption{A Seifert surface for the circle with the property that its image under  $\Sec_x$ is a hemisphere.}\label{fig:hemisphere}
\end{figure}

We now express $\map(r,x_3)$ in terms of elliptic integrals.
We use the following simplification, which follows by explicit computations.
\begin{multline}\label{eq:simplification}
\frac{r\cos t-1}{1+r^2+x_3^2-2r\cos t+x_3\sqrt{1+r^2+x_3^2-2r\cos t}}=\\
=\frac{-x_3(r\cos t-1)}{(1+r^2-2r\cos t)\sqrt{1+r^2+x_3^2-2r\cos t}}+\frac{r\cos t-1}{1+r^2-2r\cos t}.
\end{multline}
Define
\[C(r)=\int_{-\pi}^{\pi}\frac{r\cos t-1}{1+r^2-2r\cos t}dt.\]
Then
\[C(r):=\begin{cases}0\quad&\textup{if}\quad r>1\\ -\pi \quad&\textup{if}\quad r= 1 \\-2\pi \quad&\textup{if}\quad r<1 \end{cases}.\]
Using \eqref{eq:simplification} and $\cos 2\theta = 1 - 2\sin^2 \theta$ we write
\begin{multline*}
4\pi \map(r,x_3)=C(r)+
 \dfrac{2x_3}{\sqrt{(1+r)^2 + x_3^2}} \int_{0}^{\pi/2} \dfrac{dt}{\sqrt{1 - \dfrac{4r}{{(1+r)^2 + x_3^2}}\sin^2 t}} \\
- \dfrac{2x_3 (r^2 -1)}{(1+r^2)^2\sqrt{(1+r)^2 + x_3^2}} \int_{0}^{\pi/2} \dfrac{dt}{\left(1 - \dfrac{4r}{(1+r^2)}\sin^2 t\right)\sqrt{1 - \dfrac{4r}{{(1+r)^2 + x_3^2}}\sin^2 t}} =\\
\dfrac{2x_3}{\sqrt{(1+r)^2 + x_3^2}} \ellK\left(\sqrt{\dfrac{4 r}{(1+r)^2 + x_3^2}}\right) \\
+\dfrac{2x_3 (1-r)}{(1+r)\sqrt{(1+r)^2 + x_3^2}} \ellPi\left(\dfrac{4r}{(1+r)^2},\sqrt{\dfrac{4 r}{(1+r)^2 + x_3^2}}\right) + C(r).
\end{multline*}
We may write the formula in terms of Heuman's Lambda function $\ellL_0$ using the formula relating $\ellPi$ and $\ellL_0$; see
\cite[page 228]{Byrd} or \cite{Paxton}. After straightforward but tedious calculations, we obtain the following explicit formula.

\begin{proposition}[see \expandafter{\cite[Proposition 6.3.1]{Dee}}] \label{prop:elliptic_formula}
\
Let
$x=(x_1,x_2,x_3)\in \R^3$.
\begin{itemize}
\item If $x\notin \{x_1^2 + x_2^2 = 1, x_3\leq 0\}$ \quad and \quad $x_3 \neq 0$. Then
\begin{align*}
4\pi\map(r,x_3)
& = C(r) + \dfrac{2x_3}{\sqrt{(1+r)^2 + x_3^2}} \ellK(k)+\\
&+ \pi \ellL_0 \left(\arcsin \dfrac{|x_3|}{\sqrt{(1-r)^2 + x_3^2}},k\right) \dfrac{x_3(1-r)}{|x_3\|1-r|},
\end{align*}
where
\begin{equation}\label{eq:whatisk}
k=\sqrt{\frac{4r}{(1+r)^2+x_3^2}}.
\end{equation}
\item If $x_3=0$ \quad but \quad $x_1^2+x_2^2\neq 1$, then
\[\map(r,0) = \frac{C(r)}{4\pi}.\]

\item  If $x_1^2 + x_2^2 = 1 \quad and \quad x_3 < 0$, then
\[\map(1,x_3) = - \map(1,-x_3) = \frac14 + \frac{1}{4\pi}\dfrac{2x_3}{\sqrt{4 + x_3^2}} \ellK\left(\sqrt{\dfrac{4}{4 + x_3^2}}\right).\]
\end{itemize}
\end{proposition}

Another approach in computing the solid angle for an unknot was given by F. Paxton, see \cite{Paxton}. He showed that the solid angle subtended at a point $P$ with height $L$ from the unknot and with distance $r_0$ from the axis of the unknot is equal to
\begin{align*}\map=
\begin{cases}
\frac12 - \frac{1}{4\pi}\frac{2L}{R_{max}}\ellK(k) - \frac14 \ellL_0(\xi,k) \quad &\textup{if} \quad r_0 < 1 \\
\frac14 - \frac{1}{4\pi}\frac{2L}{R_{max}}\ellK(k)  \quad &\textup{if} \quad r_0 = 1 \\
- \frac{1}{4\pi}\frac{2L}{R_{max}}\ellK(k) + \frac14\ellL_0(\xi,k) \quad &\textup{if} \quad r_0 > 1
\end{cases}
\end{align*}
where $R_{max} = \sqrt{(1+r_0)^2 + L^2}$, $\xi = \arctan \frac{L}{|1-r_0|}$ and $k$ is given by \eqref{eq:whatisk}.
It can be shown that the Paxton formula agrees with the result of Proposition~\ref{prop:elliptic_formula}.

Finally we remark that the computation of the solid angle of the unknot was already studied by Maxwell. He gave the formulae
in terms of infinite series, see \cite[Chapter XIV]{Maxwell}.

\subsection{Behavior of $\map$ near $U$}\label{sec:linearbehavior}

We shall now investigate the behavior of $\map$ and its partial derivatives near $U$.

$$x_1 = 1+\varepsilon\cos2\pi\lambda,\quad x_2 = 0 \quad\textup{and} \quad x_3 = \varepsilon\sin 2\pi\lambda$$
where  $\varepsilon >0$ is small and $\lambda\in [0,1]$. We have the following result.
\begin{proposition}[see \expandafter{\cite[Proposition 6.4.2]{Dee}}]
The limit as $\varepsilon\to 0^+$ is given by
$$\lim_{\varepsilon\to 0^+} \map(1+\varepsilon\cos2\pi\lambda,0,\varepsilon\sin2\pi\lambda) = -\lambda \quad\in\mathbb R/Z.$$
\end{proposition}
\begin{proof}[Sketch of proof]
Use $r = 1 + \varepsilon\cos 2\pi\lambda$, $x_3 = \varepsilon\sin 2\pi\lambda$ and apply Proposition~\ref{prop:elliptic_formula}
together with a fact that
$$\lim_{\varepsilon\to 0^+}\dfrac{2\varepsilon\sin2\pi\lambda}{\sqrt{4+4\varepsilon\cos2\pi\lambda+\varepsilon^2}}  \ellK\left(\sqrt{\dfrac{4+4\varepsilon\cos2\pi\lambda}{4+4\varepsilon\cos2\pi\lambda+\varepsilon^2}}\right)  = 0. $$
\end{proof}
\begin{remark}
The sign of the limit is $-\lambda$ and not $+\lambda$. It is not hard to see that the orientation convention for the circle, that is, such that
$t\mapsto (\cos t,\sin t,0)$ is an oriented parametrization of $U$ 
%and $\lambda\mapsto (1+\varepsilon\cos 2\pi\lambda,0,\varepsilon\sin 2\pi\lambda)$
%gives a positive parametrization of the normal bundle to $U$ at $(1,0,0)$, 
is opposite to the convention adopted in Section~\ref{sec:sign}.
\end{remark}

Next we compute the derivatives of $\map$ near $U$. It is clear that the map $\map$ for the circle is invariant with respect to the rotational
symmetry around the $z$--axis.
Hence, if $\alpha$ is the longitudinal coordinate near $U$, then $\dfrac{\partial}{\partial\alpha}\map = 0$.
The two coordinates we have to deal with are the meridional and radial coordinates $\lambda$ and $\varepsilon$. The first
result is the following.
%% Stopped here
\begin{proposition}[see \expandafter{\cite[Proposition 6.4.3]{Dee}}]\label{prop:derivative}
We have
\[\dfrac{\partial}{\partial\varepsilon} \map(1+\varepsilon\cos2\pi\lambda,0,\varepsilon\sin2\pi\lambda) = \frac{1}{4\pi}\dfrac{2\sin2\pi\lambda( \ellK(k)-\ellE(k))}{(1+\varepsilon\cos2\pi\lambda)\sqrt{4+4\varepsilon\cos2\pi\lambda+\varepsilon^2}}.\]
\end{proposition}

We observe that as $\varepsilon \to 0^+$, we have $k\to 1^-$ by \eqref{eq:whatisk}. The numerator $\ellK(k)-\ellE(k)$ blows up, so the
right hand side of the formula in Proposition~\ref{prop:derivative} is divergent as $\varepsilon\to 0$.
For future use, we remark that by \eqref{eq:limKk} and Proposition~\ref{prop:derivative} we have
\begin{equation}\label{eq:linearr}
\left\vert\dfrac{\partial}{\partial\varepsilon} \map(1+\varepsilon\cos2\pi\lambda,0,\varepsilon\sin2\pi\lambda)\right\vert\le
C_{lin}(-\ln\varepsilon)
\end{equation}
for some constant $C_{lin}$, which can be explicitly calculated.

By Proposition~\ref{prop:derivative} the sign of $\dfrac{\partial}{\partial\varepsilon} \map$ depends on $\sin2\pi\lambda$. Hence, $\map(1+\varepsilon\cos2\pi\lambda,0,\varepsilon\sin2\pi\lambda)$
is non-decreasing with respect to $\varepsilon$ when $\lambda\in[0,\frac12]$ and it is non-increasing when $\lambda\in[\frac12,1]$.
Since we know that
$$\lim_{\varepsilon\to 0^+} \map(1+\varepsilon\cos2\pi\lambda,0,\varepsilon\sin2\pi\lambda) = -\lambda,$$
Dini's theorem, see e.g. \cite[Theorem 7.13]{Rudin}, yields that as $\varepsilon\to 0^+$,
$\map(1+\varepsilon\cos2\pi\lambda,0,\varepsilon\sin2\pi\lambda)$ converges uniformly to $-\lambda$ on $[0,1]$.
With this, the map $(\varepsilon,\lambda)\mapsto\map(1+\varepsilon\cos2\pi\lambda,0,\varepsilon\sin2\pi\lambda)$ extends to the set $\{\varepsilon=0\}$
even though $\map$ itself is not defined at $(1,0,0)$.
\begin{remark}
This extension of $\map$ through $\{\varepsilon=0\}$ will be generalized in the Continuous Extension Lemma~\ref{prop:extend1}.
\end{remark}

We now estimate the derivative of $\map$ with respect to $\lambda$.

\begin{proposition}[see \expandafter{\cite[Proposition 6.4.4]{Dee}}] \label{prop: Phi_derivative_lambda}
$$\dfrac{\partial}{\partial\lambda} \map(1+\varepsilon\cos2\pi\lambda,0,\varepsilon\sin2\pi\lambda) < 0$$
and
$$\lim_{\varepsilon\to 0^+}\dfrac{\partial}{\partial\lambda} \map(1+\varepsilon\cos2\pi\lambda,0,\varepsilon\sin2\pi\lambda)= -1.$$
\end{proposition}

\begin{proof}
Set
\begin{align*}
k &= \sqrt{\dfrac{4+4\varepsilon\cos2\pi\lambda}{4+4\varepsilon\cos2\pi\lambda + \varepsilon^2}}\\
k' &= \sqrt{1-k^2} = \dfrac{\varepsilon}{\sqrt{4+4\varepsilon\cos2\pi\lambda + \varepsilon^2}}.
\end{align*}

Note that
\begin{align*}
\dfrac{\partial k}{\partial\lambda} &= 
%\dfrac{\pi}{k}\left(\dfrac{(4+4\varepsilon\cos2\pi\lambda+\varepsilon^2)(-4\varepsilon\sin2\pi\lambda) - (4+4\varepsilon\cos2\pi\lambda)(-4\varepsilon\sin2\pi\lambda)}{(4+4\varepsilon\cos2\pi\lambda+\varepsilon^2)^2}\right)\\
2\pi\dfrac{-k'^3\sin2\pi\lambda}{\sqrt{1+\varepsilon\cos2\pi\lambda}}\\
\intertext{and}
\dfrac{\partial k'}{\partial\lambda} &= \frac{-k}{\sqrt{1-k^2}} \frac{\partial k}{\partial\lambda} = -\frac{k}{k'} \frac{\partial k}{\partial\lambda}.
\end{align*}
Using \eqref{dHeuman/dk} and \eqref{dHeuman/d(angle)} we have
\begin{align*}
&\dfrac{\partial}{\partial\lambda} \map(1+\varepsilon\cos2\pi\lambda,0,\varepsilon\sin2\pi\lambda) \\
&= \frac{1}{4\pi}
\dfrac{\partial}{\partial\lambda} (2\sin2\pi\lambda k' \ellK(k) \pm \pi\ellL_0 (\arcsin|\sin2\pi\lambda|,k))\\
&= \frac{1}{2\pi}\sin 2\pi\lambda\left(\ellK(k)\frac{\partial k'}{\partial\lambda} + k'\frac{\partial \ellK(k)}{\partial\lambda}\right) + k'\ellK(k)\cos 2\pi\lambda\\
&\pm \frac{1}{4}\left(\frac{\partial}{\partial(\arcsin|\sin2\pi\lambda|)}\ellL_0 (\arcsin|\sin2\pi\lambda|,k))\right) \frac{\partial (\arcsin|\sin2\pi\lambda|)}{\partial\lambda}\\
&\pm \frac{1}{4}\left(\frac{\partial}{\partial k}\ellL_0 (\arcsin|\sin2\pi\lambda|,k))\right) \frac{\partial k}{\partial\lambda}\\
%&= \frac{1}{2\pi}\sin 2\pi\lambda \frac{\partial k}{\partial\lambda}\left(-\frac{kK(k)}{k'} + \frac{E(k) - k'^2K(k)}{kk'}\right)\\
&= \frac{1}{2\pi}\sin2\pi\lambda\dfrac{\partial k}{\partial\lambda}\left(\dfrac{\ellE(k)- \ellK(k)}{kk'}\right) + k' \ellK(k)\cos2\pi\lambda \\
&- \left(\dfrac{\ellE(k) - k'^2\sin^2 2\pi\lambda  \ellK(k))}{\sqrt{1-k'^2\sin^2 2\pi\lambda}}\right) - \frac{1}{2\pi}\left(\dfrac{(\ellE(k)- \ellK(k))\sin 2\pi\lambda\cos 2\pi\lambda}{k \sqrt{1-k'^2\sin^2 2\pi\lambda}}\right)\dfrac{\partial k}{\partial\lambda}.
\end{align*}
See also \cite[Equation (6.11) and Proposition 6.4.3]{Dee}. As $\varepsilon\to 0^+$, we have $k\to 1$, $k'\to 0$.
Since $k' \ellK(k)\to 0$ and $k'\ellE(k)\to 0$ as $\varepsilon\to 0^+$,
the only significant term in the above expression is
$-\dfrac{\ellE(k)}{\sqrt{1-k'^2\sin^2 2\pi\lambda}}$, hence
$$\lim_{\varepsilon\to 0^+}\dfrac{\partial}{\partial\lambda} \map(1+\varepsilon\cos 2\pi\lambda,0,\varepsilon\sin 2\pi\lambda) = -\ellE(1) = -1.$$
\end{proof}

We have estimated the derivatives of $\map$ with respect to $\varepsilon$ and $\lambda$. We can now give the following corollary, which is a
straightforward consequence of \eqref{eq:linearr}.
\begin{corollary}\label{cor:clin}
The derivatives $\frac{\partial}{\partial x_j}\map(x_1,x_2,x_3)$, $j=1,2,3$ have at most a logarithmic pole at points $(x_1,x_2,x_3)$ close to
$U$. More precisely, there exists a constant $C_{circ}$ such that
\[\left|\frac{\partial}{\partial x_j}\map(x_1,x_2,x_3)\right|\le C_{circ}(-\ln dist((x_1,x_2,x_3),U)).\]
\end{corollary}

We remark that from \eqref{eq:pullbackofeta}, we get much weaker estimates on the derivative.
We do not know if these weaker estimates can be improved for general manifolds $M$.

To conclude, we show level sets of the function $(r,x_3)\mapsto\map(r,x_3)$ for the circle in Figure~\ref{fig:unknotmap}. Notice that in the Figure
the half-lines
stemming from point $(1,0)$ (and not parallel to the $x_3=0$ line) intersect infinitely many level sets near the point $(1,0)$.
This suggests that the radial derivative $\frac{\partial}{\partial\varepsilon}\map(1+\varepsilon\cos2\pi\lambda,\varepsilon\sin2\pi\lambda)$ is unbounded
as $\varepsilon\to 0^+$. We proved this fact rigorously in Proposition~\ref{prop:derivative}.
\begin{figure}
\includegraphics[scale=0.5]{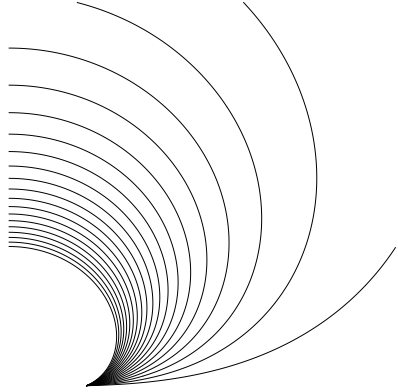}
\caption{Level sets of the function $(r,x_3)\mapsto\map(r,x_3)$ for the circle.}\label{fig:unknotmap}
\end{figure}
\section{Derivatives of $\map$ near $M$}\label{sec:behavior}

We begin by recalling a well-known fact in differential geometry.
\begin{proposition}\label{prop:aux1}
Let $X\subset\R^{n+2}$ be a $k$-dimensional, smooth, compact submanifold with smooth boundary. Then, there exists a constant $C_X$ such that for every $x\in\R^{n+2}$
and for any $r>0$ we have
\[\vol_k(X\cap B(x,r))\le C_X r^k.\]
Moreover, increasing $C_X$ if necessary, we may assume that if $\omega$ is a $k$--form on $\R^{n+2}$ whose coefficients are bounded
by $T$, then $|\int_{X\cap B(x,r)}\omega|\le C_X T r^k$.
\end{proposition}
The result is well-known to the experts, therefore we present only a sketchy proof.
\begin{proof}[Sketch of proof]
Let $\delta_k$ be the volume of unit $k$--dimensional ball.
By smoothness of $X$ we infer that $\lim_{r\to 0}\displaystyle\frac{\vol_k(X\cap B(x,r))}{r^k}$ is $0$, $\frac12\delta_k$ or $\delta_k$
depending on whether $x\notin X$, $x\in\partial X$ or $x\in X\setminus\partial X$.
Using Vitali's theorem,
one shows that there exists $r_0$ independent of $x$
such that $\displaystyle\frac{\vol_k(X\cap B(x,r))}{r^k}\le 2\delta_k$ for all $r<r_0$. We take $C_X$ to be the maximum
of $2\delta_k$ and $\vol_k(X)/r_0^k$.

The second part is standard and left to the reader.
\end{proof}

\subsection{The Separation Lemma}
The form $\eta_z$ used  in Section~\ref{sec:pullback} has a pole at $z\in S^{n+1}$. In the applications for given $x\in\R^{n+2}\setminus M$
we choose a point $z$ such that $z\notin\Sec_x(M)$. Such a point exists, see Remark~\ref{rem:znotinM}  above. However, in order to obtain
a meaningful bound for $\Sec_x^*\eta_z$, we need to know that $z$ is separated from $\Sec_x(M)$, in the sense
that there exists a constant $D$ such that $\langle y,z\rangle\le D$ for any $y\in\Sec_x(M)$. In this
section we show that the constant $D<1$ can be chosen independently of $x$.

\begin{lemma}[Separation Lemma]\label{lem:aux2}
There exist $\varepsilon_0>0$ and $D<1$ such that the set $$N_0=(M+B(0,\varepsilon_0))\setminus M$$
of points at distance less than $\varepsilon_0$
from $M$ (and not lying in $M$)
can be covered by a finite number of open sets $U_1,\ldots, U_l$ with the following property: for each $i$ there exists
a point $z_i\in S^{n+1}$ such that for any $x\in U_i$
we have $\Sec_x(M)\subset \{u\in S^{n+1}:\langle u,z_i\rangle\le D\}$.
\end{lemma}
\begin{remark}
In general it is impossible for a given point $x\in M$ to find an element $z\in S^{n+1}$ and a neighborhood $U\subset\R^{n+2}$ of $x$, such that
for every $x'\in U$ we have $z\notin\Sec_{x'}(M)$. In fact, the opposite holds. For any $z\in S^{n+1}$ the sequence $x_n=x-\frac{z}{n}$ has
the property that $z\in\Sec_{x_n}(M)$ and $x_n\to x$. This is the main reason why the proof of an apparently obvious lemma is not trivial.

Put differently, the subtlety of the proof of Lemma~\ref{lem:aux2} lies in the fact
that the image $\Sec_x(M)$ can be defined for $x\in M$ as a closure of $\Sec_x(M\setminus\{x\})$, 
but we cannot argue that $\Sec_x(M)$ depends continuously on $x$, if $x\in M$.
\end{remark}
\begin{proof}[Proof of Lemma~\ref{lem:aux2}]
Take a point $x\in M$. Let $V$ be the affine subspace tangent to $M$ at $x$, that is $V=x+T_xM$. The image
$\Sec_x(V\setminus \{x\})$ is the intersection
\[S_x:=T_xM\cap S^{n+1}.\]
\begin{lemma}\label{lem:neartangent}
For any open subset $U\subset S^{n+1}$ containing $S_x$, there exists $r>0$ such that the $\Sec_x(M\cap B(x,r)\setminus\{x\})$ is contained in $U$.
\end{lemma}
\begin{proof}
Suppose the contrary, that is, for any $n$ there exists a point $y_n\in M$ such that $\|x-y_n\|<\frac{1}{n}$ and $\Sec_x(y_n)\notin U$.
In particular $y_n\to x$. As $M$ is a smooth submanifold of $\R^{n+2}$, the tangent space $T_xM$ is the linear space of limits of secant lines through $x$.
This means that if $y_n\to x$ and $y_n\in M$, then, up to passing to a subsequence, $\Sec_x(y_n)$ converges to a point in $S_x$. But then, starting with
some $n_0>0$, we must have $\Sec_x(y_n)\in U$ for all $n>n_0$. Contradiction.
\end{proof}
Choose now a neighborhood $U$ of $S_x$ and $r$ from Lemma~\ref{lem:neartangent}. As $S_x$ is invariant with respect to the symmetry $y\mapsto -y$,
we may and will assume that $U$ also is. We assume that $U$ is small neighborhood of $S_x$, but in fact we will only need that $U$ is not dense in $S^{n+1}$.
We will need the following technical result.
\begin{lemma}\label{lem:symmetry}
Suppose $v\in S^{n+1}\setminus U$.
Then there exists an open neighborhood $W_v\subset\R^{n+2}$ of $x$, such that if $x'\in W_v$, then either $v\notin \Sec_{x'}(M\cap \ol{B(x,r)})$
or $-v\notin\Sec_{x'}(M\cap \ol{B(x,r)})$.
\end{lemma}
\begin{proof}[Proof of Lemma~\ref{lem:symmetry}]
We act by contradiction. Assume the statement of the lemma does not hold. That is,
there is a sequence $x_n$ converging to $x$ such that both $v$ and $-v$ belong to $\Sec_{x_n}(M\cap\ol{B(x,r)})$. This means that for any $n$
the line $l_{x_n}:=\{x_n+tv,t\in\R\}$ intersects $M\cap\ol{B(x,r)}$ in at least one point for $t>0$ and at least one point for $t<0$. For each $n$, choose a point $y_{n}^+$ in
$M\cap\ol{B(x,r)}\cap l_{x_n}\cap\{t>0\}$ and a point $y_n^-$ in $M\cap\ol{B(x,r)}\cap l_{x_n}\cap\{t<0\}$. 
In particular $\Sec_{x_n}(y_n^+)=v$ and $\Sec_{x_n}(y_n^-)=-v$.

By taking subsequences of $\{y^+_n\}$ and $\{y^-_n\}$ we can assume that $y_n^+\to y^+$
and $y_n^-\to y^-$ for some $y^+,y^-\in M\cap\ol{B(x,r)}$.

If $y^+\neq x$, then the line $l_x=\{x+tv\}$ passes through $y^+$, but this means that $v\in\Sec_x(M\cap \ol{B(x,r)})$, but the assumption was that $v\notin U$,
so we obtain a contradiction. So $y^+=x$. Analogously we prove that $y^-=x$.

Finally suppose $y^+=y^-=x$. The line $l_x=\{x+tv\}$ is the limit of secant lines passing through $y_n^+$ and $y_n^-$, therefore $l_x$ is tangent
to $M$ at $x$. But then $v\in S_x\subset U$. Contradiction.
\end{proof}
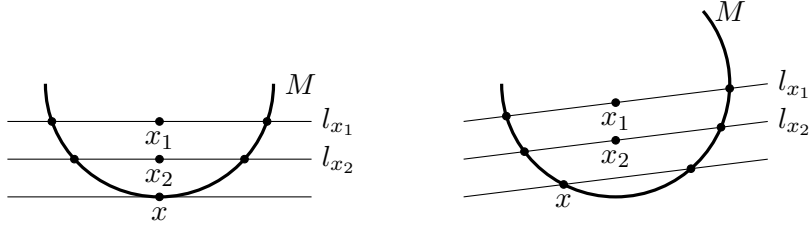
\begin{figure}
\begin{tikzpicture}
\path[name path=M,very thick,draw] (-4,1.5) arc (180:360:1.5cm) node [right] {$M$};
\path[name path=lx1,thin,draw] (-4.5,1) -- ++(4,0) node[right]{$l_{x_1}$};
\path[name path=lx2,thin,draw] (-4.5,0.5) -- ++(4,0) node[right]{$l_{x_2}$};
\path[name path=lx3,thin,draw] (-4.5,0) -- ++(4,0);
\path[name intersections={of=M and lx1,by={v1,v2}}];
\path[name intersections={of=M and lx2,by={v3,v4}}];
\draw[color=black,fill=black] (v1) circle (0.05);
\draw[color=black,fill=black] (v2) circle (0.05);
\draw[color=black,fill=black] (v3) circle (0.05);
\draw[color=black,fill=black] (v4) circle (0.05);
\draw[color=black,fill=black] (-2.5,0) circle (0.05) node[below] {$x$};
\draw[color=black,fill=black] (-2.5,0.5) circle (0.05) node[below] {$x_2$};
\draw[color=black,fill=black] (-2.5,1) circle (0.05) node[below] {$x_1$};
\begin{scope}[xshift=6cm]
\path[name path=aM,very thick,draw] (-4,1.5) arc (180:400:1.5cm) node [right] {$M$};
\path[name path=alx1,thin,draw] (-4.5,1) -- ++(4,0.5) node[right]{$l_{x_1}$};
\path[name path=alx2,thin,draw] (-4.5,0.5) -- ++(4,0.5) node[right]{$l_{x_2}$};
\path[name path=alx3,thin,draw] (-4.5,0) -- ++(4,0.5);
\path[name intersections={of=aM and alx1,by={av1,av2}}];
\path[name intersections={of=aM and alx2,by={av3,av4}}];
\path[name intersections={of=aM and alx3,by={av5,av6}}];
\draw[color=black,fill=black] (av1) circle (0.05);
\draw[color=black,fill=black] (av2) circle (0.05);
\draw[color=black,fill=black] (av3) circle (0.05);
\draw[color=black,fill=black] (av4) circle (0.05);
\draw[color=black,fill=black] (av5) circle (0.05) node[below] {$x$};
\draw[color=black,fill=black] (av6) circle (0.05);
\draw[color=black,fill=black] (-2.5,0.75) circle (0.05) node[below] {$x_2$};
\draw[color=black,fill=black] (-2.5,1.25) circle (0.05) node[below] {$x_1$};
\end{scope}
\end{tikzpicture}
\caption{Proof of the Lemma~\ref{lem:symmetry}. To the left: a sequence of lines $l_{x_1},l_{x_2}$ converges to a line that is tangent
to $M$ at $x$, so that $v\in T_xM$. To the right: a sequence of lines $l_{x_1},l_{x_2}$ converges to a line that passes through $x$
and intersects $M$ at some point. Then $v\in\Sec_x(M)$.}\label{fig:aux2}
\end{figure}
We extend the argument of Lemma~\ref{lem:symmetry} in the following way.
\begin{lemma}\label{lem:symmetry2}
Suppose $v\in S^{n+1}\setminus U$. Then there exist an open set $V_{loc}\subset S^{n+1}$ containing $v$ and
an open ball $B_{loc}\subset B(x,r)$ containing $x$ such that if $x'\in B_{loc}$ then there exists a choice of sign $\eta\in\{\pm 1\}$
possibly depending on $x$ such that $\eta V_{loc}$ is disjoint from the image $\Sec_{x'}(M\cap\ol{B(x,r)})$.
\end{lemma}
\begin{proof}
The proof is a modification of the proof of Lemma~\ref{lem:symmetry}. We leave
the details for the reader.
\end{proof}
Resuming the proof of Lemma~\ref{lem:aux2} define the sets $B_{loc+}$ and $B_{loc-}$ by
\[B_{loc\pm}=\{x'\in B_{loc}\colon \textit{for every }v'\in V_{loc}\textit{ we have }\pm v'\notin\Sec_{x'}
(M\cap\ol{B(x,r)})
\}.\]
Then clearly $B_{loc+}\cup B_{loc-}=B_{loc}$.
\begin{lemma}\label{lem:openness}
The subsets $B_{loc+}\setminus M$ and $B_{loc-}\setminus M$ are open subsets of $B_{loc}\setminus M$.
\end{lemma}
\begin{proof}
The lemma follows from the fact that $M\cap\ol{B(x,r)}$ is closed and $\Sec_{x'}$ is continuous with respect to $x'$ as
long as $x'\notin M$.
\end{proof}

The next step in the proof of the Separation Lemma~\ref{lem:aux2} is the following.
\begin{lemma}
There exists an open set $B_x\subset\R^{n+2}$ containing $x$, a point $v_x\in S^{n+1}$ and an open set $V_x\subset S^{n+1}$ containing $v_x$
such that if $x'\in B_x$ and $v'\in V_x$ then either $v'$ or $-v'$ does not belong to $\Sec_{x'}(M)$.

Moreover, there are two subsets $B_x^{\pm}$
of $B_x$ such that $B_x^+\cup B_x^-=B_x$, $B_x^\pm$ are open in $B_x\setminus M$ and if $x'\in B_x^\pm$ and $v'\in V_x$, then
$\Sec_{x'}(M)$ does not contain $\pm v'$.
\end{lemma}
\begin{proof}
Let $U$ and $r>0$
be as in the statement of Lemma~\ref{lem:neartangent}. The set $\Sec_x(M\setminus B(x,r))$ is the image of the compact manifold $M\setminus B(x,r)$
of dimension $n$ under a smooth map, hence its interior is empty. Therefore there exists a point $v\in S^{n+1}$ such that
neither $v$ nor $-v$ is in the image $\Sec_x(M\setminus B(x,r))$ and also neither $v$ nor $-v$ is in $U$. By the continuity of $\Sec_x$, there exist
a small ball $B_{gl}\subset B(x,r)$ with center $x$ and a small ball $V_{gl}\subset S^{n+1}$ containing $v$ such that
if $v'\in V_{gl}$ and $x'\in B_{gl}$, then neither $v'$ nor $-v'$ belongs to $\Sec_{x'}(M\setminus B(x,r))$.
Let $V_{loc}$ and $B_{loc}$ be from Lemma~\ref{lem:symmetry2}. Define $V_x=V_{loc}\cap V_{gl}$ and $B_x=B_{loc}\cap B_{gl}$. Then for any $x'\in B_x$
and $v'\in V_x$ we have that either $v'\notin\Sec_{x'}M$ or $-v\notin\Sec_{x'}$.

We define $B_x^\pm$ as intersections of $B_{loc\pm}$ with $B_x$, where $B_{loc\pm}$ are as in Lemma~\ref{lem:openness} above.
\end{proof}
We resume the proof of the Separation Lemma~\ref{lem:aux2}. Define
\[\delta_x=\sup_{y\in S^{n+2}\setminus V_x} \langle v_x,y\rangle.\]
As $V_x$ is an open set containing $v_x$, we have $\delta_x<1$. This means that
if $x'\in B_x$ and $y\in\Sec_x(M)$, then either $\langle y,v_x\rangle\le\delta_x$ or $\langle y,-v_x\rangle\le\delta_x$.

Cover now $M$ by open sets $B_x$ for $x\in M$. As $M$ is compact, there exists a finite set $x_1,\ldots,x_n$
such that $M\subset B_{x_1}\cup\dots\cup B_{x_n}$. The compactness of $M$ implies also that there exists
$\varepsilon_0>0$ such that the set $M+B(0,\varepsilon_0)$, that is, the set of points at distance less than $\varepsilon_0$ from $M$,
is contained in $B_{x_1}\cup\dots\cup B_{x_n}$.
Define $D=\max(\delta_{x_1},\ldots,\delta_{x_n})$ and let $N_0=(M+B(0,\varepsilon_0))\setminus M$. For $i=1,\ldots,n$ define
$v_i=v_{x_i}$ and $B_i^\pm=B_{x_i}^\pm\cap N_0$. Then $B_i^\pm$ cover $N_0$ and for any $i$, if $x'\in B_i^\pm$ and $y\in\Sec_{x'}(M)$,
then $\langle y,\pm v_i\rangle\le D$ as desired.
\end{proof}

For points $x$ that are at distance greater than $\varepsilon_0$ from $M$, the statement of the Separation Lemma~\ref{lem:aux2} holds as well.
\begin{theorem}[Separation Theorem]\label{thm:aux2}
There exists a constant $D<1$ such that for any $x\in \R^{n+2}\setminus M$ there exists an open neighborhood $U_x\subset\R^{n+2}$ of $x$
and a point $z\in S^{n+2}$, such that for any $x'\in U_x$ and $y\in\Sec_{x'}(M)$ we have $\langle y,z\rangle<D$.
\end{theorem}
\begin{proof}
Denote by $D_{close}$ the constant $D$ from the Separation Lemma~\ref{lem:aux2}. The constant $D_{close}$ works for points at distance less than $\varepsilon_0$
from $M$.

We work with points far from $M$.
Choose $R>0$ large enough so that $M\subset B(0,R)$. For any $x\notin B(0,R)$ we can take the point $\frac{x}{\|x\|}$ for $z$
and then if $y\in\Sec_x(M)$, then $\langle y,z\rangle\le 0$. By the continuity of $x\mapsto\Sec_x$ we may choose a neighborhood
$W_x$ of $x$ such that if $x'\in W_x$ and $y_n\in \Sec_{x'}(M)$, then $\langle y,z\rangle$ is bounded from above by a small positive number, say $\frac{1}{10}$.
This takes care
of the exterior of the ball $B(0,R)$. We define $D_{far}=\frac{1}{10}$. The constant $D_{far}$ works for points outside of the ball $B(0,R)$. 

Let $P=\{x\in \ol{{B}(0,R)}\colon \dist(x,M)\ge \varepsilon_0\}$.
For any point $x\in P$, $\Sec_x(M)$ is the image of an $n$-dimensional compact manifold under a smooth map, so it is a
boundary closed subset of $S^{n+1}$. Thus there exist a point $z_x\in S^{n+1}$ and a neighborhood of $U_x$ of $z_x$ such that
$U_x\cap\Sec_x(M)=\emptyset$. Shrinking $U_x$ if necessary we may guarantee that there exists a neighborhood $W_x\subset\R^{n+2}$ of $x$ such
that if $x'\in W_x$ and $y\in\Sec_{x'}(M)$, then $y\notin U_x$. We define again
\[\delta_x=\sup_{y\in S^{n+2}\setminus V_x}\langle z_x,y\rangle<1.\]

The sets $W_x$ cover $P$ and we take a finite subcover $W_{x_1},\ldots,W_{x_M}$. We define $D_{mid}$ as the maximum of $\delta_{x_1},\ldots,\delta_{x_M}$.
The constant $D_{mid}$ works for points at distance between $\varepsilon_0$ and inside of $B(0,R)$.

It is enough to take $D=\max(D_{close},D_{mid},D_{far})$.
\end{proof}

From now on we assume that $D<1$ is fixed.

\subsection{The Drilled Ball Lemma}
We begin to bound the value of $\frac{\partial}{\partial x_i}\map(x)$. To this end we will differentiate the coefficients
of $\Sec_x^*\eta_z$. The point $z$ will always be chosen in such a way that $\langle \Sec_{x'}(y),z\rangle<D$ for all
$y\in M$ and for all $x'$ sufficiently close to $x$.

The next result estimates the contribution to $\frac{\partial\map}{\partial x_j}$ from integrating $\Sec_x^*\eta$ on a drilled ball.
\begin{lemma}\label{lem:importantestimate}
Suppose $\alpha,\beta\in(0,1)$ and $x\in \R^{n+2}$.
Fix $\varepsilon>0$ and  define $$M_{\alpha\beta\varepsilon} := M\cap (B(x,\varepsilon^\beta)\setminus B(x,\varepsilon^\alpha)).$$
Then, for any $i=1,\ldots,n+2$
\[\left|\frac{\partial}{\partial x_i}\int_{M_{\alpha\beta\varepsilon}} \Sec_x^*\eta\right|<C_{drill}\varepsilon^\gamma,\]
where $\gamma=n\beta-(n+1)\alpha$ and $C_{drill}=C_MC^D_{n,1}$ is independent of $\alpha$,  $\beta$ and $x$.
\end{lemma}
\begin{proof}
By Lemma~\ref{lem:mainbound} and the Separation Lemma~\ref{lem:aux2}, the derivative of the pull-back $\frac{\partial}{\partial x_i}\Sec_x^*\eta$ is an $n$--form
whose coefficients are bounded from above by
by $\frac{C^D_{n,1}}{\|y-x\|^{n+1}}$. If $y\in M_{\alpha\beta\varepsilon}$, then $\|y-x\|\ge\varepsilon^{\alpha}$.
The form $\frac{\partial}{\partial x_i}\Sec_x^*\eta$ is integrated over $M_{\alpha\beta\varepsilon}$. We use Proposition~\ref{prop:aux1} twice.
First to conclude that the  volume of $M_{\alpha\beta\varepsilon}$ is bounded from above by
$C_M\varepsilon^{n\beta}$, second to conclude
that the integral is bounded by $C^D_{n,1}C_M\varepsilon^{n\beta-(n+1)\alpha}$.
\end{proof}

The next result shows that if $M$ is locally parametrized by some $\Psi$ then if we take a first order approximation, the contribution
to the derivative of $\map(x)$ from the local piece does not change much. We need to set up some assumptions.

Choose $\varepsilon>0$ and $\alpha\in(\frac12,1)$. For a fixed point $x$ at distance  $\varepsilon$ from $M$
we set
$M_{\alpha\varepsilon}=M\cap \ol{B(x,\varepsilon^\alpha)}$.
We assume that $\varepsilon,\alpha$ are such that $M_{\alpha\varepsilon}$ can parametrized by % a picture is needed
\[\Psi:B' \to M_{\alpha\varepsilon},\]
where $B'$ is some bounded open subset in $\R^n$ and $\Psi(0)$ is the point on $M_{\alpha\varepsilon}$ that is nearest to $x$.
We also assume that $B'$ is a star-shaped, that is, if $w\in B'$, then $tw$ also in $B'$ for $t\in[0,1]$. 
For simplicity of the formulae we may transform $B'$ in such a way that
\begin{equation}\label{eq:dpsi0}
D\Psi(0)=\begin{pmatrix} 1&0&\ldots&0\\ 0&1&\ldots&0\\ \vdots&\vdots&\ddots&\vdots\\ 0&0&\ldots&1\\ 0&0&\ldots&0\\ 0&0&\ldots&0
\end{pmatrix}.\end{equation}
Choose $\sigma>0$ in such a way that $B'$ is a subset of  an $n$-dimensional ball $B(0,\sigma)$  and
$B'$ is not a subset of $B(0,\sigma/2)$.
Let $\Psi_1$ be the first order approximation of $\Psi$, that is $\Psi_1(w)=\Psi(0)+D\Psi(0)w$.
Let $M_1$ be the image of $B'$ under $\Psi_1$. Write $C_1$ and $C_2$ for the supremum of the first
and second derivative of $\Psi$ on $B'$.

\begin{lemma}[Approximation Lemma]\label{lem:local}
Suppose
\begin{equation}\label{eq:locall}
\varepsilon<\min\left((4C_2)^{-1/\alpha}, 2^{-1/(\alpha-1)},(32C_2)^{-1/(2\alpha-1)}\right)
\end{equation}
and $\alpha\ge \frac12$.
There exists a constant $C_{app}$ depending on $\Psi$ such that
\[\left|\frac{\partial}{\partial x_j}\int_{M_{\alpha\varepsilon}}\Sec_{x}^*\eta-
\frac{\partial}{\partial x_j}\int_{M_1}\Sec_{x}^*\eta\right|\le C_{app}\varepsilon^{\delta},\]
where $\delta=\alpha(n+3)-(n+2)$.
\end{lemma}
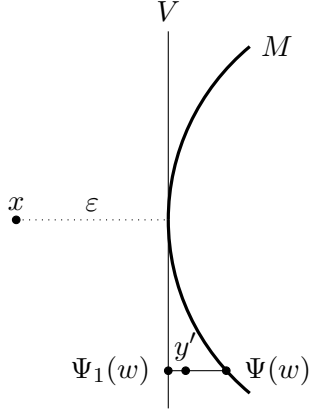
\begin{figure}
\begin{tikzpicture}
\draw[name path=newV](0,-2.5) -- (0,2.5) node[above] {$V$};
\path[name path=newM,draw,very thick](0,0) arc (180:130:3cm) node[right] {$M$};
\path[name path=newM2,draw,very thick](0,0) arc (180:230:3cm);
\draw[fill=black](-2,0) circle (0.05) node[above]{$x$};
\path[name path=newL] (-3,-2) -- (4,-2);
\path[name intersections={of=newV and newL,by={w1}}];
\path[name intersections={of=newM2 and newL,by={w2}}];
\draw[fill=black](w1) circle(0.05) node[left=1mm]{$\Psi_1(w)$};
\draw[fill=black](w2) circle(0.05) node[right=1mm]{$\Psi(w)$};
\draw[thin] (w1)--(w2);
\draw[fill=black] ($(w1)!0.3!(w2)$) circle (0.05) node[above]{$y'$};
\draw[thin,dotted] (-2,0) -- (0,0);
\draw(-1,0.2) node {$\varepsilon$};
\end{tikzpicture}
\caption{Notation of the proof of the Approximation Lemma~\ref{lem:local}.}
\end{figure}
\begin{proof}
As our first step we relate
$\sigma$ with $\varepsilon$ and $\alpha$.
\begin{lemma}\label{lem:sigma}
We have $\sigma\le 4\varepsilon^\alpha$.
\end{lemma}
\begin{proof}[Proof of Lemma~\ref{lem:sigma}]
Suppose $w\in B(0,\sigma)$, $w\neq 0$. Write $\wt{w}$ for the vector in $\R^{n+2}$ given by
$\frac{1}{\|w\|}(w,0,0)$. Define a function $\Psi_w\colon \R\to\R$ by the formula $\Psi_w(t)=\langle\Psi(t\frac{w}{\|w\|}),\wt{w}\rangle$.
By the definition we have $\Psi_w(0)=0$. From \eqref{eq:dpsi0} we calculate that $\frac{d}{dt}\Psi_w(0)=1$. Furthermore, as the second
derivative of $\Psi$ is bounded by $C_2$
we have that $|\frac{d^2}{dt^2}\Psi_w(t)|<C_2$. It follows that $\Psi_w(t)\ge t-\frac{C_2}{2}t^2$. 
Set $t_0=2\varepsilon^\alpha$. As $\varepsilon^\alpha<\frac{1}{4C_2}$ by the assumptions,
we have $\Psi_w(t_0)\ge 2\varepsilon^\alpha-2C_2\varepsilon^{2\alpha}\ge\frac32\varepsilon^\alpha$.
Clearly $\Psi_w(t_0)\le\|\Psi(t_0\frac{w}{\|w\|})\|$, hence
$\|\Psi(t_0\frac{w}{\|w\|})\|\ge\frac32\varepsilon^\alpha$.

The condition $\varepsilon<2^{-1/(\alpha-1)}$ implies that $\frac12\varepsilon^\alpha>\varepsilon$.
By the
triangle inequality
\[\|\Psi(t_0\frac{w}{\|w\|})-x\|\ge \|\Psi(t_0\frac{w}{\|w\|})\|-\|x\|>\varepsilon+\varepsilon^\alpha-\varepsilon=\varepsilon^\alpha.\]
This means that 
$\Psi(t_0\frac{w}{\|w\|})$ cannot possibly belong to $B(x,\varepsilon^\alpha)$, hence it is not in the image $\Psi(B')=M\cap B(x,\varepsilon^\alpha)$.
This shows that $t_0\frac{w}{\|w\|}$ cannot belong to $B'$. As $w$ was an arbitrary point in $B'$, this implies that no element in $B'$
can have norm $2\varepsilon^\alpha$. As $B'$ is connected, this implies that $B'$ must be contained in $B(0,2\varepsilon^\alpha)$.
By
the definition of $\sigma$ we immediately recover that $\sigma\le 4\varepsilon^\alpha$.
\end{proof}

%%%%
We resume the proof of Lemma~\ref{lem:local}.
Choose $w\in B(0,\sigma)$.
Write $y_0=\Psi(w)$, $y_1=\Psi_1(w)$.
By the Taylor formula we have 
\begin{equation}\label{eq:psitaylor}
\|y_0-y_1\|=\|\Psi_1(w)-\Psi(w)\|\le C_2\|w\|^2. 
\end{equation}
We have $\|w\|<4\varepsilon^\alpha$ and $2\alpha>1$.
Using the assumption that $\varepsilon<(32C_2)^{-1/(2\alpha-1)}$ we infer that $C_2(4\varepsilon^\alpha)^2\le\frac12\varepsilon$.
so that $\|\Psi_1(w)-\Psi(w)\|\le\frac12\varepsilon$.
Therefore, as $\dist(x,M_{\alpha\varepsilon})=\varepsilon$, we infer that for each point $y'$ in the interval connecting $y_0$ and $y_1$
we have $\|x-y'\|\ge\frac12\varepsilon$.

Write now for $i=1,\ldots,n+1$:
\begin{equation}\label{eq:partialSec}
\frac{\partial \Sec_x^*\eta}{\partial x_i}=\sum_{i\neq j} F_{ij}(x,y)dy_1\wedge\ldots\widehat{dy_i,dy_j}\ldots\wedge dy_{n+2}.
\end{equation}
By the mean value theorem, for any $i,j$ there exists a point $y'$ in the interval connecting $y_0$ and $y_1$ such that
\begin{equation}\label{eq:meanvalue}
F_{ij}(x,y_0)-F_{ij}(x,y_1)=\|y_0-y_1\|\partial_vF_{ij}(x,y'),
\end{equation}
where $\partial_v$ is the directional derivative in the direction of the vector $\frac{y_0-y_1}{\|y_0-y_1\|}$. Now $\|x-y'\|\ge\frac12\varepsilon$,
hence by Lemma~\ref{lem:technicalestimate}:
\begin{equation}\label{eq:partialv}\left|\partial_vF_{ij}(x,y')\right|\le C^D_{n,2}\left(\frac12\varepsilon\right)^{-n-2}.\end{equation}
From \eqref{eq:psitaylor} we deduce that
\begin{equation}\label{eq:aboutFij}
\left|F_{ij}(x,\Psi(w))-F_{ij}(x,\Psi_1(w))\right|\le C\|w\|^2\varepsilon^{-n-2}
\end{equation}
for some constant $C$ depending on $C_2$.

Set $G_{ij}(w)$ and $H_{ij}(w)$ to be defined by
\begin{align*}
\Psi^*dy_1\wedge\ldots\widehat{dy_i,dy_j}\ldots\wedge dy_{n+2}&=G_{ij}dw_1\wedge\ldots\wedge dw_n\\
\Psi_1^*dy_1\wedge\ldots\widehat{dy_i,dy_j}\ldots\wedge dy_{n+2}&=H_{ij}dw_1\wedge\ldots\wedge dw_n.\\
\end{align*}
The values of $G_{ij}$ and $H_{ij}$ are bounded by a constant depending on $C_1$. Moreover
the expression $D\Psi(w)-D\Psi_1(w)$ has all entries bounded from above by $\|w\|$ times a constant, hence an exercise in linear algebra shows that
\begin{equation}\label{eq:CG}
|G_{ij}(w)-H_{ij}(w)|\le C_G\|w\|
\end{equation}
for some constant $C_G$ depending on $C_1$. Now write
\begin{multline*}
\Psi^*F_{ij}(x,y)dy_1\wedge\ldots\widehat{dy_i,dy_j}\ldots\wedge dy_{n+2}-\Psi_1^*F_{ij}(x,y)dy_1\wedge\ldots\widehat{dy_i,dy_j}\ldots\wedge dy_{n+2}
=\\
\left(F_{ij}(x,\Psi(w))G_{ij}(w)-F_{ij}(x,\Psi_1)(w)H_{ij}(w)\right)dw_1\wedge\ldots\wedge dw_n.
\end{multline*}
We estimate using \eqref{eq:aboutFij}:
\begin{multline*}
|F_{ij}(x,\Psi(w))G_{ij}-F_{ij}(x,\Psi_1)H_{ij}|\le \\
|F_{ij}(x,\Psi(w))-F_{ij}(x,\Psi_1(w))|\cdot |H_{ij}(w)|+|F_{ij}(x,\Psi(w)|\cdot|G_{ij}(w)-H_{ij}(w)|.
\end{multline*}
Combining this with \eqref{eq:aboutFij} we infer that
\begin{equation}\label{eq:new_estimate}
|F_{ij}(x,\Psi(w))G_{ij}-F_{ij}(x,\Psi_1)H_{ij}|\le C(\|w\|^2\varepsilon^{-n-2}+\|w\|\varepsilon^{-n-1}),
\end{equation}
where the factor $\varepsilon^{-n-1}$ comes from the estimate of $F_{ij}(x,\Psi(w))$ and the constant $C$
depends on previous constants, that is, $C$ depends on $C_1$, $C_2$.

We use now \eqref{eq:new_estimate} together with \eqref{eq:partialSec} and the definitions of $G_{ij}$, $H_{ij}$. After straightforward
calculations we obtain for some constant $C$:
\[\left|\int_{B'}\frac{\partial}{\partial x_j}\left(\Psi^*\Sec_x(y)-\Psi_1^*\Sec_x(y)\right)\right|\le\int_{B(0,\sigma)} C(\|w\|^2\varepsilon^{-n-2}+\|w\|\varepsilon^{-n-1}).\]
The last expression is bounded by $C_{app}(\varepsilon^{(n+3)\alpha-(n+2)}+\varepsilon^{(n+2)\alpha-(n+1)})$, where $C_{app}$
is a new constant. As $\alpha<1$
and $\varepsilon\ll 1$, the term $\varepsilon^{(n+3)\alpha-(n+2)}$ is dominating.
\end{proof}

In the following result we show that the constants in the Approximation Lemma~\ref{lem:local} can be made universal, that is, depending
only on $M$ and $\alpha$ and not on $x$ and $\varepsilon$.
\begin{proposition}\label{prop:local}
For any $\alpha\in(\frac12,1)$ there exist constants $C_\alpha$ and $\varepsilon_1>0$ such that for any $x\notin M$ such that
$\dist(x,M)<\varepsilon_1$ we have
\[\left|\frac{\partial\map}{\partial x_i}(x)-\frac{\partial\map_V}{\partial x_i}(x)\right|\le C_{\alpha}\varepsilon^{(n+2)\alpha-(n+1)}.\]
Here $\map_V$ is a map $\map$ defined relatively to the plane $V$ that is tangent to $M$ at a point $y$ such that $\dist(x,M)=\|x-y\|$.
\end{proposition}
\begin{proof}

Cover $M$ by a finite number of subsets $U_i$ such that each of these subsets can be parametrized by a map $\Psi_i\colon V_i\to U_i$,
where $V_i$ is a bounded subset of $\R^n$. By the compactness of $M$, there exists $\varepsilon_1>0$ such that if $U\subset M$ has diameter
less than $\varepsilon_1$, then $U$ is contained in one of the $U_i$. Shrinking $\varepsilon_1$ if necessary we may and will assume that
if $\dist(x,M)<\varepsilon_1$, then there is a unique point $y\in M$ such that $\dist(x,M)=\|x-y\|$.

Set $C_1$ and $C_2$ to be the upper bound on the first and the second derivatives of all of the $\Psi_i$. The derivative $D\Psi_i(w)$ is injective
for all $w\in V_i$. We assume that $C_0>0$ is such that $\|D\Psi_i(w)v\|\ge C_0\|v\|$ for all $v\in\R^n$, $i=1,\ldots,n$ and $w\in V_i$.

Choose a point $x$ at distance $\varepsilon>0$ to $M$ such that $2\varepsilon^\alpha<\varepsilon_1$ and $\varepsilon<\varepsilon_1$. Let
$M_{\alpha\varepsilon}=B(x,\varepsilon^\alpha)\cap M$. As this set has diameter less than $\varepsilon_1$, we infer that
$M_{\alpha\varepsilon}\subset U_i$ for some $i$. Let $B=\Psi_i^{-1}(M_{\alpha\varepsilon})\subset V_i$. Let $y\in M$ be the unique point
realizing $\dist(x,M)=\|x-y\|$. We translate the set $B$ in such a way that $\Psi_i(0)=y$. Next we rotate the coordinate system in $\R^{n+2}$
in such a way that the image of $D\Psi(0)$ has a block structure $A\oplus\left(\begin{smallmatrix} 0 & 0\\ 0& 0\end{smallmatrix}\right)$
for some invertible matrix $A$. We know that $A^{-1}$ is a matrix with coefficients bounded by an universal constant depending
on $c_1$ and $C_1$. Define now $B_0=A^{-1}(B)$ and $\Psi_x=\Psi_i\circ A$. Then $\Psi$ has first and second derivatives bounded by
a constant depending on $C_1,C_2$ and $c_1$. Denote these constants by $C_1(x),C_2(x)$. Let also be $C_0(x)>0$ be such that if $w,w'\in B$,
then $\|\Psi_x(w)-\Psi_x(w')\|\ge C_0(x)\|w-w'\|$. Such constant exists because $D\Psi(w)$ is injective and we use the mean value theorem.
Moreover $C_0(x)$ is bounded below by a constant depending on $C_1$, $C_2$ and $C_0$.

It remains to ensure that the following two conditions are satisfied. First, the set $B=\Psi_x^{-1}(M_{\alpha\varepsilon})$ has to be star-shaped,
second the inequality \eqref{eq:locall} is satisfied. The second condition is obviously guaranteed by taking $\varepsilon_1$ sufficiently small. 
We claim that the first condition can also be guaranteed by taking $\varepsilon_1$. To see this we first notice that if $\varepsilon_1$ is sufficiently
small, then $M_{\alpha\varepsilon}$ is connected for all $\varepsilon<\varepsilon_1$. Next, we take a closer look at the definition of $B\subset V_i$.
Namely we can think of $B$ as the set of points $w\in V_i$
satisfying the inequality $R(w)\le\varepsilon^\alpha$, where
\[R(w)=\|\Psi_x(w)-x\|^2=\langle \Psi_x(w)-x,\Psi_x(w)-x\rangle.\]
For $v\in\R^{n}$ we have
\[D^2R(0)(v,v)=2\langle D\Psi_x(0)(v,v),\Psi_x(0)-x\rangle+\langle D\Psi_x(0)v,D\Psi_x(0)v\rangle.\]
By the construction of $\Psi_x$ we have that $\langle D\Psi_x(0)v,D\Psi_x(0)v\rangle=\|v\|^2$ hence
\[D^2R(0)(v,v)\ge (1-2\varepsilon C_2(x))\|v\|^2.\]
Generalizing this for $w\in B$ and $v\in\R^n$ we have
\[D^2R(w)(v,v)= \|DR(w)v\|^2+2\langle D^2\Psi_x(w)(v,v),\Psi_x(w)-x\rangle.\]
Now $\|D^2\Psi_x(w)(v)\|\le C_2(x)\|v\|^2$ and by the mean value theorem also $\|DR(w)-DR(0)\|\le C_2(x)\|w\|$. Suppose $\|x-\Psi_x(w)\|\le\varepsilon^\alpha$.
Then $\|\Psi_x(w)-\Psi_x(0)\|\le 2\varepsilon^\alpha$ and so $\|w\|\le 2C_0(x)\varepsilon^\alpha$. Hence
\[D^2R(w)(v,v)\ge (1-2\varepsilon^\alpha C_0(x)C_2(x)-2\varepsilon^\alpha C_2(x))\|v\|^2.\]
This shows that $R$ is a convex function if $\varepsilon$ is sufficiently small. Hence $B$ is a convex subset, in particular, it is also star-shaped.
Therefore all the assumptions of the Approximation Lemma~\ref{lem:local} are satisfied, the statement follows.
\end{proof}

\subsection{Approximation Theorem}
Combining the Drilled Ball Lemma~\ref{lem:importantestimate} and Proposition~\ref{prop:local}, we obtain a result which is the main technical estimate.
% change the exponent!
\begin{theorem}[Approximation Theorem]\label{thm:maintechnical}
For any $\theta\in(\frac{n+2}{n+4},1)$ there exists a constant $C_\theta$ such that if $x$ is at distance $\varepsilon>0$ to $M$
and $\varepsilon<\varepsilon_1$,
$y_0\in M$ is the point realizing the minimum of $dist(x,M)$ and $V$ is the tangent space to $M$ passing through $y_0$,
then for any $j=1,\ldots,n+2$
\[\left|\frac{\partial}{\partial x_j}\map(x)-\frac{\partial}{\partial x_j}\map_V(x)\right|\le C_\theta\varepsilon^{-\theta}.\]
Here, $\map_V$ is the map $\map$ defined relatively to the hyperplane $V$.
\end{theorem}
\begin{remark}
In Section~\ref{sec:linear}, we have shown that $\map_V$ is not well defined, but if we restrict to `Seifert hypersurfaces' for $V$
which are half-spaces (and that is what we in fact do), then $\map_V$ is defined up to an overall constant. In particular, its derivatives do not
depend on the choice of the half-space.
\end{remark}
\begin{proof}
Let $\xi=\frac{\partial}{\partial x_j}\Sec_x^*\eta$. 
Set also $\alpha_0=\frac{n+2}{n+3}-\frac{\theta}{n+3}$. We have $(n+3)\alpha_0-(n+2)= -\theta$, hence by Proposition~\ref{prop:local} we obtain.
\begin{equation}\label{eq:local2}
\left|\int_{M\cap B(x,\varepsilon^{\alpha_0})}\xi-\int_{V\cap B(x,\varepsilon^{\alpha_0})}\xi\right|\le C_{app}\varepsilon^{-\theta}.
\end{equation}
Set $\alpha_{k+1}=\alpha_k+\frac1n(\alpha_k-\theta)$. By the assumptions
we have $\alpha_0<\theta$, so $\alpha_{k+1}<\alpha_k$ and the sequence $\alpha_k$ diverges to $-\infty$. 
Suppose $k_0<\infty$ is the first index, when $\alpha_{k_0}\le 0$. Set $\alpha_{k_0}=0$
in this case. We use repeatedly the Drilled Ball Lemma~\ref{lem:importantestimate} for $\beta=\alpha_{k+1}$ and $\alpha=\alpha_k$, $k=0,\ldots,k_0-1$.
We are allowed
to do that because for $k<k_0-1$ we have
\[n\alpha_{k+1}-(n+1)\alpha_k=n\left(\frac{n+1}{n}\alpha_k-\frac1n\theta\right)-(n+1)\alpha_k=-\theta\]
and $n\alpha_{k_0}-(n+1)\alpha_{k_0-1}\ge -\theta$.
%\[n\left(\frac{n+1}{n+3}-\frac{\delta(k+1)}{n}\right)-(n+1)\left(\frac{n+1}{n+3}-\frac{\delta}{k}{n}\right)=-\frac{n+1}{n+3}-\frac{\delta}{n}(k-n)\ge -\theta.\]
Summing the inequality from the Drilled Ball Lemma~\ref{lem:importantestimate} for $k$ from $0$ to $k_0-1$ we arrive at
\begin{equation}\label{eq:local3}
\left|\int_{M\cap B(x,\varepsilon^{\alpha_{k_0}})\setminus B(x,\varepsilon^{\alpha_0})}\xi\right|\le k_0\cdot C_{drill}\varepsilon^{-\theta}.
\end{equation}
Recall that $\alpha_{k_0}=0$. Equation~\eqref{eq:local3} does not cover the part of $M$ outside of $B(x,1)$. However, on $M\setminus B(x,1)$, the form $\xi$
is easily seen to have coefficients bounded above by a constant independent of $\varepsilon$ and $x$, hence
\begin{equation}\label{eq:local35}
\left|\int_{M\setminus B(x,1)}\xi\right|\le C_{ext},
\end{equation}
for some constant $C_{ext}$ depending on $M$ but not on $x$ and $\varepsilon$.
It remains to show
\begin{equation}\label{eq:local4}
\left|\int_{V\cap B(x,\varepsilon^{\alpha_0})}\xi-\int_{V}\xi\right|\le C_{flat}\varepsilon^{-\theta}.
\end{equation}
We cannot use the Drilled Ball Lemma~\ref{lem:importantestimate} directly, because $V$ is unbounded. However, we will use similar ideas as in the proof of
the Drilled Ball Lemma~\ref{lem:importantestimate}. The form $\xi$ is an $n$-form whose coefficients on $V$ are bounded by $\nobreak{C^{D'}_{n,1}\|y-x\|^{-(n+1)},}$ where
$D'$ is such that $\pi_x^{-1}(V)\subset\{u_{n+2}<D'\}$. Its restriction
to $V$ is equal to some function $F_x(y)$ times the volume form on $V$, where $|F_x(y)|\le C_VC^{D'}_{n,1}\|y-x\|^{-(n+1)}$ (it is easy to see
that as $V$ is a half-plane, $C_V$ exists). Therefore we need to bound
\begin{equation}\label{eq:to_be_integrated_over_radial}
\int_{V\setminus B(x,\varepsilon^{\alpha_0})} \|y-x\|^{-(n+1)}.
\end{equation}
The method is standard. Introduce radial coordinates on $V$ centered at $y_0$ and notice that $\|y-x\|\le 2\|y_0-y\|$ as long as
$y\in V\setminus B(x,\varepsilon^{\alpha_0})$. Perform first the integral \eqref{eq:to_be_integrated_over_radial} over radial coordinates obtaining
the integral over the radius only, that is
\begin{multline*}
\int_{V\setminus B(x,\varepsilon^{\alpha_0})} \|y-x\|^{-(n+1)}\le \int_{\varepsilon^{0}}^\infty 2^{n-1}\sigma_{n-1}r^{n-1}\cdot r^{-n-1}dr=\\
=2^{n-1}\varepsilon^{-\alpha_0}\sigma_{n-1}\le 2^{n-1}\sigma_{n-1}\varepsilon^{-\theta},
\end{multline*}
where $\sigma_{n-1}$ is the volume of a unit sphere of dimension $n-1$.
This proves \eqref{eq:local4} with $C_{flat}=2^{n-1}\sigma_{n-1}C_V C^{D'}_{n,1}$.

Combining \eqref{eq:local2}, \eqref{eq:local3},\eqref{eq:local35} and \eqref{eq:local4} we obtain the desired statement.
\end{proof}

\subsection{The main estimate for the derivative}\label{sec:mainestimate}
This section extends the intuitions given in Section~\ref{sec:linearbehavior}.

It is a result of Erle \cite{Erle}
that the normal bundle of $M \subset \R^{n+2}$ is trivial, but there might be many different trivializations, one class for each element of $[M,SO(2)]=[M,S^1]=H^1(M)$. Choose a pair of two normal
vectors $v_1,v_2$ on $M$ such that at each point $y\in M$, $v_1(y)$ and $v_2(y)$ form an oriented orthonormal basis of the normal space $N_yM$.
Choose $\varepsilon_0<\varepsilon_1$ and let $N$ be the tubular neighborhood
of $M$ of radius $\varepsilon_0$. By taking $\varepsilon_0>0$ sufficiently small
we may and will assume that each $y'\in N$ can be uniquely written as $y+t_1v_1+t_2v_2$ for $y\in M$ and $t_1,t_2\in\R$.

Let $y\in M$. Choose a local coordinate system $w_1,\ldots,w_n$ in a neighborhood of $y$ such that
$\left|\frac{\partial w_j}{\partial x_i}\right|\le C_w$ for some constant $C_w$.
The local coordinate system $w_1,\ldots,w_n$ on $M$ induces a local coordinate system on $N$ given by $w_1,\ldots,w_n,t_1,t_2$. Let $r,\phi$ be such
that $t_1=r\cos(2\pi \phi)$, $t_2=r\sin(2\pi\phi)$.
\begin{theorem}[Main Estimate Theorem]\label{prop:nextbound}
For $\theta\in\left(\frac{n+2}{n+4},1\right)$,
we have $\left|\frac{\partial\map}{\partial w_j}\right|\le C_wC_\theta r^{-\theta}$ and
$\left|\frac{\partial\map}{\partial r}\right|\le C_\theta r^{-\theta}$. Moreover
$\left|\frac{\partial\map}{\partial\phi}-\eps\right|\le C_{\theta}r^{1-\theta}$,
where $\eps\in\{\pm 1\}$ depending on the orientation of $M$.
\end{theorem}
\begin{proof}
Choose a point $x=(w_1,\ldots,w_n,t_1,t_2)$. Let $y_0=(w_1,\ldots,w_n,0,0)$ be the point minimizing the distance
from $x$ to $M$.
We will use the Approximation Theorem~\ref{thm:maintechnical}.
So let $V$ be the $n$-dimensional plane tangent to $M$ at $y_0$. The map $\map_V$ is the map $\map$ relative to $V$. By the
explicit
calculations in Section~\ref{sec:linear}
we infer that $\frac{\partial\map_V}{\partial w_j}(x)=\frac{\partial\map_V}{\partial r}(x)=0$ for $j=1,\ldots,n$
and $\frac{\partial\map_V}{\partial\phi}=\eps$. Now $\frac{\partial\map}{\partial w_j}$ differs from the derivatives
of $\map_V$ by at most $C_\theta r^{-\theta}$ by the Approximation Theorem~\ref{thm:maintechnical}.

On the other hand, by chain rule $\frac{\partial\map}{\partial \phi}=-r\sin\phi\frac{\partial\map}{\partial t_1}+r\cos\phi\frac{\partial\map}{\partial t_2}$.
Applying Theorem~\ref{thm:maintechnical} we infer that
$\left|\frac{\partial\map}{\partial\phi}(x)-\frac{\partial\map_V}{\partial\phi}(x)\right|\le C_\theta r^{1-\theta}$. The same argument
shows that $\left|\frac{\partial\map}{\partial r}(x)\right|\le C_\theta r^{-\theta}$. Notice that the derivatives with respect to $r$ and $\theta$
do not depend on $C_w$: this is so because the length of the framing vectors $v_1$ and $v_2$ is $1$.
\end{proof}

\section{Behavior of $\map$ near $M$}\label{sec:behavior2}
Throughout the section we choose $\theta\in\left(\frac{n+2}{n+4},1\right)$. The constant $\varepsilon_0$ is as defined in Section~\ref{sec:mainestimate}.
We decrease further $\varepsilon_0$ to ensure that
\begin{equation}\label{eq:cthet}
C_\theta\varepsilon_0^{1-\theta}<\frac12
\end{equation}
so that, by Theorem~\ref{prop:nextbound}
\begin{equation}\label{eq:dphi}
\left|\frac{\partial\map}{\partial\phi}-\eps\right|<\frac12.
\end{equation}
\subsection{Local triviality of $\map$ near $M$}
Let now $X=(0,\varepsilon_0]\times S^1\times M$. Let $\Pi\colon X\to N\setminus M$ be a parametrization given by
\[\Pi\colon (r,\phi,x)\mapsto x+v_1r\cos\phi+v_2r\sin\phi.\]
The composition $\map\circ\Pi\colon X\to S^1$ will still be denoted by $\map$. We are going to show that this map is a locally trivial
fibration whose fibers have bounded $(n+1)$-dimensional volume.

\begin{lemma}[Fibration Lemma]\label{cor:locallytrivial}
The map $\map\colon X\to S^1$ is a smooth, locally trivial fibration, whose fiber is $(0,\varepsilon_0]\times M$.
\end{lemma}
\begin{proof}
Choose $r\in(0,\varepsilon_0]$ and $x\in M$. Consider the map $\map_{r,x}\colon S^1\to S^1$ given by $\map_{r,x}=\map|_{\{r\}\times S^1\times \{x\}}$.
The derivative of $\map_{r,x}$ is equal to $\frac{\partial\map}{\partial\phi}$, by \eqref{eq:dphi} it belongs either to the interval $(-\frac32,-\frac12)$ 
or to $(\frac12,\frac32)$ depending
on the $\epsilon$. It follows that $\map_{r,x}$
is a diffeomorphism. In particular, given $r\in(0,\varepsilon_0]$ and $x\in M$, 
for any $t\in S^1$, there exists a unique point $\Theta_t(r,x)$ such that $\map(r,\Theta_t(r,x),x)=t$.
In this way we get a bijection $\Theta_t(r,x)\colon (0,\varepsilon_0]\times M\to\Phi^{-1}(t)$.

Again by \eqref{eq:dphi} $|\frac{\partial\map}{\partial\phi}|>\frac12>0$, so 
by the implicit function theorem we infer that $\Theta_t$ is in fact a smooth map. Then $\Theta_t$ is a smooth parametrization of the fiber
of $\map$. It remains to show that $\map$ is locally trivial.

To this end we choose a point $t\in S^1$ and let $U\subset S^1$ be a neighborhood of $t$. Define the map $\wt{\Theta}\colon (0,\varepsilon_0]\times U\times M\to
\map^{-1}(U)$ by the formula
\[\wt{\Theta}(r,t,x)=(r,\Theta_t(r,x),x).\]
Clearly $\wt{\Theta}$ is a bijection.
As $\Theta_t$ depends smoothly on the parameter $t$, we infer that $\wt{\Theta}$ is a smooth map and the map $\map^{-1}(U)\to (0,\varepsilon_0]\times U\times M$
given by $(r,\phi,x)\mapsto (r,\map(r,\phi,x),x)$ is its inverse. Therefore $\wt{\Theta}$ is a local trivialization.
\end{proof}
\begin{remark}\label{rem:otherfiber}
Define the map $\map_M\colon X\to S^1\times M$ by $\map_M(r,\phi,x)=(\map(r,\phi,x),x)$. The same argument as in the proof of 
Fibration Lemma~\ref{cor:locallytrivial} shows that $\map_M$ is a locally trivial fibration with fiber $(0,\varepsilon_0]$. For given $(t,x)\in S^1\times M$
the map $r\mapsto (r,\Theta_t(r,x),x)$ parametrizes the fiber over $(t,x)$.
\end{remark}
As a consequence of Fibration Lemma~\ref{cor:locallytrivial} we show that $\map\colon\R^{n+2}\setminus M\to S^1$ does not have too many critical points.
This is a consequence of Sard's theorem and
the control of $\map$ near $M$ provided by Lemma~\ref{cor:locallytrivial}.
\begin{proposition}
The set of critical values of $\map\colon\R^{n+2}\setminus M\to S^1$ is a closed boundary set of measure zero.
\end{proposition}
\begin{proof}
Extend $\map$ to a map from $S^{n+2}\setminus M\to S^1$ as in Corollary~\ref{cor:extendstosphere}.
We split the $S^{n+2}\setminus M$ as a union of $S^{n+2}\setminus N$ and $N\setminus M$,
where, recall, $N$ is the set of points at distance less than or equal to $\varepsilon_0$. By Sard's theorem
the map $\map$ restricted to $\ol{S^{n+2}\setminus N}$ has
a set of critical points which is closed boundary and of measure zero. On the other hand, on $N\setminus M$ the map has no critical points at all, because
by the Fibration Lemma~\ref{cor:locallytrivial} the map $\map$ restricted to $N\setminus M\cong X$ is a locally trivial fibration.
\end{proof}

We conclude by showing the following control of the fibers of $\map\colon X\to S^1$:
\begin{lemma}[Bounded Volume Lemma]\label{prop:boundedarea}
There exists a constant $A>0$ such that for any $t\in S^1$  the $(n+1)$-dimensional volume of
$\map^{-1}(t)\cap X$ is bounded from above by $A$.
\end{lemma}
\begin{proof}
Parametrize  $\map^{-1}(t)\cap X$ by $\Theta_t\colon (0,\varepsilon_0]\times M\to X$
as in the proof of the Fibration Lemma~\ref{cor:locallytrivial} above. 
Choose local coordinate system on an open subset $Y\subset M$
and let $X_r^{loc}=[r,\varepsilon_0]\times S^1\times Y$. We aim to show that $\vol_{n+1}\map^{-1}(t)\cap X_r^{loc}$ is bounded
by a constant independent of $r$. Write
\begin{equation}\label{eq:area}
\vol_{n+1}\map^{-1}(t)\cap X_r^{loc}=\int_{[r,\varepsilon_0]\times Y}\sqrt{1+{\Theta'}_{w_1}^2+\ldots+{\Theta'}_{w_n}^2+{\Theta'}_{r}^2}\, dl_{n+1},
\end{equation}
where $dl_{n+1}$ is the $(n+1)$-dimensional Lebesgue measure on $[r,\varepsilon_0]\times Y$, we write $\Theta$ for $\Theta_t$ and
$\Theta'_z$ is a shorthand for $\frac{\partial\Theta}{\partial z}$, and $z$ is any variable of $\{r,w_1,\ldots,w_n\}$.
By the implicit function theorem $\Theta'_z=-\frac{\partial\map}{\partial z}\left(\frac{\partial\map}{\partial\phi}\right)^{-1}$.
Equation~\eqref{eq:dphi} implies that $\left|\frac{\partial\map}{\partial\phi}\right|\ge\frac12$, then
\begin{equation}\label{eq:boundfortheta}\begin{split}
|\Theta'_{w_j}|&\le 2C_\theta C_w r^{-\theta},\ \ j=1,\ldots,n\\
|\Theta'_{r}|&\le 2C_\theta r^{-\theta}.
\end{split}\end{equation}
Hence
\[\sqrt{1+{\Theta'}_{w_1}^2+\ldots+{\Theta'}_{w_n}^2+{\Theta'}_{r}^2}\le\sqrt{1+4C^2_\theta+4(n+1)C^2_wC^2_\theta}r^{-\theta}\le C_\Theta r^{-\theta},\]
where $C_\Theta$ is a constant. Thus
\[\vol_{n+1}\map^{-1}(t)\cap X_r^{loc}\le C_\Theta(\varepsilon_0^{1-\theta}-r^{1-\theta})\vol_n Y\le C_\Theta\varepsilon_0^{1-\theta}\vol_nY.\]
Now $M$ being compact can be covered by a finite number of coordinate neighborhoods, we sum up all the contributions to get
\[\vol_{n+1}\map^{-1}(t)\cap X\le C_\Theta\varepsilon_0^{1-\theta}\vol_nM.\]
\end{proof}

\begin{remark}
Bounded Volume Lemma~\ref{prop:boundedarea} shows that the volume of the fibers $\map^{-1}(t)$ is bounded near $M$ by a constant that does not depend on $t$.
This does not generalize to bounding a global volume of $\map^{-1}(t)$: one can show that the  volume of $\map^{-1}(0)$ is infinite using 
Corollary~\ref{cor:extendstosphere}.
\end{remark}
\subsection{Extension to of $\map$ through $r=0$}
We pass to study the closure of the fibers of map $\map^{-1}(t)\cap X$. This is done by extending the map $\map$.
Set
\[\ol{X}=[0,\varepsilon_0]\times S^1\times M.\]
The manifold $\ol{X}$ can be regarded as an analytic blow-up of the neighborhood $N\setminus M$.
\begin{lemma}[Continuous Extension Lemma]\label{prop:extend1}
The map $\map\colon X\to S^1$ extends to a continuous map $\ol{\map}\colon\ol{X}\to S^1$.
\end{lemma}
\begin{proof}
Let $f_r\colon S^1\times M\to S^1$ be given by $f_r(\phi,x)=\map(r,\phi,x)$. We shall show that as $r\to 0$ the functions $f_r$ converge uniformly.
The limit, $f_0$, will be the desired extension.

We use Proposition~\ref{prop:nextbound}. In fact, choose $r_0,r_1\in(0,\varepsilon_0]$. Then for $(\phi,x)\in S^1\times M$ we have
\[|f_{r_0}(\phi,x)-f_{r_1}(\phi,x)|\le\int_{r_0}^{r_1} \left|\frac{\partial}{\partial r}f_r(\phi,x)\right| dr\le \int_{r_0}^{r_1} C_\theta r^{-\theta}dr=\frac{C_{\theta}}{1-\theta}
(r_0^{1-\theta}-r_1^{1-\theta}).\]
As $r\mapsto r^{1-\theta}$ is a uniformly continuous function taking value $0$ at $0$, we obtain that $f_r$ uniformly converge to some limit, which we call $f_0$. This
amounts to saying that $\map$ extends to a continuous function on $\ol{X}$.
\end{proof}
\begin{remark}\label{rem:otherfiber2}
Consider the map $\map_M$ defined in Remark~\ref{rem:otherfiber}. Then the proof of Continuous Extension Lemma~\ref{prop:extend1} generalizes to showing
that the map $\map_M$ extends to the continuous map $\ol{\map}_M\colon \ol{X}\to S^1\times M$.
\end{remark}
We can also calculate the function $\map$ for $r=0$.
\begin{proposition}\label{prop:onmap}
There exists a continuous function $\rho\colon M\to S^1$ such that $\ol{\map}(0,\phi,x)=\epsilon\phi+\rho(x)\bmod 1$.
\end{proposition}
\begin{proof}
By Main Estimate Theorem~\ref{prop:nextbound} we have that for any $c>0$ there exists $r_c>0$ such that if $r\in(0,r_c)$, then for $\phi,\phi'\in S^1$
and $x\in M$:
\begin{equation}\label{eq:lipschitz}
(1-c)|\phi-\phi'|\le |\map(r,\phi,x)-\map(r,\phi',x)|\le (1+c)|\phi-\phi'|.
\end{equation}
As $\map(r,\phi,x)$ converges uniformly to $\map(0,\phi,x)$, we infer that \eqref{eq:lipschitz} holds for $r=0$ and arbitrary $c>0$. This means that
actually
\[|\ol{\map}(0,\phi,x)-\ol{\map}(0,\phi',x)|=|\phi-\phi'|.\]
This is possible only if $\ol{\map}(0,\phi,x)=\ol{\map}(0,0,x)\pm \phi\bmod 1$, where the sign is equal to $\epsilon$.
We set $\rho(x)=\ol{\map}(0,0,x)$.
\end{proof}
Recall from the Fibration Lemma~\ref{cor:locallytrivial} that
$$\Theta_t\colon (0,\varepsilon_0]\times M\to \map^{-1}(t)$$
is a diffeomorphism.
\begin{theorem}\label{thm:graph}
For any $t\in S^1$, the maps $\Theta_t\colon (0,\varepsilon_0]\times M\to \map^{-1}(t)$ extend to a continuous map
$\ol{\Theta}_t\colon [0,\varepsilon_0]\times M\to \ol{\map}^{-1}(t)\subset \ol{X}$.
The map $\ol{\Theta}_t$ is injective.
\end{theorem}
\begin{proof}
By \eqref{eq:boundfortheta} $\frac{\partial\Theta_t}{\partial r}$
is bounded by $2C_\theta r^{-\theta}$, so we the same argument as in the proof of Lemma~\ref{prop:extend1}
shows that $\Theta_{t}(r,x)$ converges as $r\to 0$ uniformly with respect to $x$. Therefore $\ol{\Theta}_{t}$ is well defined.

The composition $\map_M\circ\Theta_t\colon(0,\varepsilon_0]\times M\to (0,\varepsilon_0]\times M$ is an identity. 
%For fixed $r\in (0,\varepsilon_0]$
%the restriction of this composition to $\{r\}\times M$ is also an identity. As $r\to 0$, the limit of the identity map is the identity map, 
Hence
$\ol{\map}_M\circ\ol{\Theta_t}\colon [0,\varepsilon_0]\times M\to [0,\varepsilon_0]\times M$ is also an identity. In particular, $\ol{\Theta}_t$
is injective.
\end{proof}

We next prove the surjectivity of $\ol{\Theta}_t$. % Surjectivity is not obvious. I can draw a picture, modelling a problem.
Before we state the proof, we indicate a possible problem in Figure~\ref{fig:possibleproblem}.
\begin{figure}
\begin{tikzpicture}[y=2pt,x=2pt,yscale=-1]
  \path[draw=red!50!black]
    (45,111) -- (45,151);
  \path[draw=red!50!black]
 (45,130) -- (114,130);
  \draw[thick] (40,180) -- (120,180);
  \draw (45,177) -- (45,183);
  \draw (45,185) node [scale=0.7] {$r=0$};
  \draw (114,177) -- (114,183);
  \draw (114,185) node [scale=0.7] {$r=\varepsilon_0$};
  \draw (140,100) -- (140,160);
  \fill[color=red!50!black] (140,130) circle (1);
  \path[draw=blue!50!black, thin]
    (45,107.6395) .. controls (48.5030,107.2254) and (50.9045,108.1432) ..
    (51.8503,110.9804) .. controls (52.7491,113.6767) and (54.9441,122.6835) ..
    (57.4629,123.9430) .. controls (60.5004,125.4617) and (106.7058,126.8829) ..
    (114,126.8829);
  \fill[color=blue!50!black] (140,126.8829) circle (0.8);
  \path[draw=blue!80!black,thin]
    (45,104.2987) .. controls (47.4738,104.2601) and (51.3945,104.0828) ..
    (52.7857,106.1696) .. controls (54.7792,109.1597) and (56.8460,118.5701) ..
    (60.1356,119.6667) .. controls (72.8104,123.8916) and (100.3337,123.0075) ..
    (114,123.0075);
  \fill[color=blue!80!black] (140,123.0075) circle (0.8);
  \path[draw=blue,thin]
    (45,101.3587) .. controls (48.1269,102.2204) and (54.0950,101.7325) ..
    (56.5275,104.1650) .. controls (60.1598,107.7973) and (60.0858,114.8392) ..
    (65.7483,116.7267) .. controls (78.8876,121.1065) and (100.4386,117.0324) ..
    (114,118.5976);
  \fill[color=blue] (140,118.5976) circle (0.8);
  \path[draw=green!30!black,thin]
    (45,153.1422) .. controls (49.9618,151.0136) and (50.2877,145.7771) ..
    (51.5162,143.3200) .. controls (52.7447,140.8629) and (51.9252,137.7911) ..
    (55.0575,135.7028) .. controls (59.3004,132.8743) and (105.6850,135.0347) ..
    (114,135.0347);
  \fill[color=green!30!black] (140,135.0347) circle (0.8);
  \path[draw=green!50!black,thin]
    (45,155.0799) .. controls (48.7078,154.9064) and (51.6413,152.9503) ..
    (53.7212,150.8704) .. controls (57.6967,146.8948) and (56.8197,139.2443) ..
    (60.5365,139.1773) .. controls (64.2534,139.1104) and (97.6092,138.8960) ..
    (114,138.7764);
  \fill[color=green!50!black] (140,138.7764) circle (0.8);
  \path[draw=green!80!black,thin]
    (45,159.6235) .. controls (50.7151,158.1624) and (55.1299,155.0511) ..
    (59.6679,150.6699) .. controls (64.2060,146.2888) and (62.6783,145.1894) ..
    (66.4165,143.5873) .. controls (70.1547,141.9852) and (114.4679,142.3990) ..
    (114,141.9837);
  \fill[color=green!90!black] (140,141.9837) circle (0.8);
  \path[draw=green!50!red,thin]
    (45,164.3675) .. controls (48.3863,164.1617) and (61.4834,152.6880) ..
    (72.4969,150.9372) .. controls (83.5103,149.1864) and (97.7742,146.8978) ..
    (114,146.6609);
  \fill[color=green!50!red] (140,146.6609) circle (0.8);
  \path[draw=blue!50!red,thin]
    (45,98.0178) .. controls (47.8522,97.8568) and (50.7296,98.1441) ..
    (53.5875,98.5524) .. controls (57.2013,99.0686) and (68.5517,108.3600) ..
    (73.0982,110.1786) .. controls (85.7745,115.2491) and (100.3999,114.9895) ..
    (114,114.9895);
  \fill[color=blue!50!red] (140,114.9895) circle (0.8);
\draw[dashed,->] (120,130) -- node[midway, above, scale=0.7] {$\map$} (135,130);
\draw[dashed,->] (80,175) --node[midway, right,scale=0.7] {$\Theta_t$} (80,155);
\end{tikzpicture}
\caption{The picture indicates the necessity of proving the surjectivity of $\ol{\Theta}_t$, the map $\ol{\Theta}_t$
is not onto. In Theorem~\ref{thm:graph2} we show that the situation as on the picture cannot happen.}\label{fig:possibleproblem}
\end{figure}
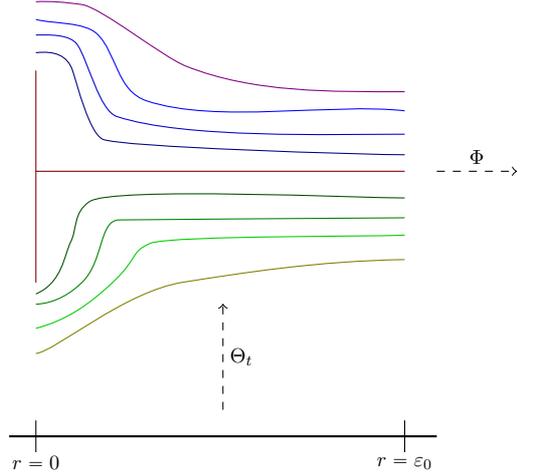
\begin{theorem}\label{thm:graph2}
The map $\ol{\Theta}_t$ is onto $\ol{\map}^{-1}(t)$.
\end{theorem}
\begin{proof}
By the Fibration Lemma~\ref{cor:locallytrivial} the map $\Theta_t$ is onto $\map^{-1}(t)\subset X$. Hence it is enough
to show that $\ol{\Theta_t}|_{\{0\}\times M}$ is onto $\ol{\map}^{-1}(t)\cap(\ol{X}\setminus X)$.

Observe that by Proposition~\ref{prop:onmap} the intersection
$\ol{\map}^{-1}(t)\cap\{0\}\times S^1\times\{x\}$ consist of one point for any $x\in M$ and $t$.

%We are first going to show that for any $x\in M$ and any $t\in S^1$ the intersection
%$\ol{\map}^{-1}(t)\cap (\{0\}\times S^1\times\{x\})\subset \ol{X}$ consists of a single point.
%To see this notice that if $\phi,\phi'\in S^1$ and $(r,x)\in(0,\varepsilon_0]\times X$, then by \eqref{eq:dphi} we have
%\[\left|\map(r,\phi,x)-\map(r,\phi',x)\right|\ge \frac12|\phi-\phi'|\]
%Moreover, taking $r$ small enough we can replace the factor $\frac12$ on the right hand side by a number arbitrary close to $1$. By passing to a limit $r\to 0$
%we obtain that
%\[\left|\ol{\map}(0,\phi,x)-\ol{\map}(0,\phi',x)\right|\ge c|\phi-\phi'|\]
%for $c$ arbitrary close to $1$. This means that $\ol{\map}$ restricted to $r=0$ is injective and 
%$\ol{\map}^{-1}(t)\cap (\{0\}\times S^1\times\{x\})\subset \ol{X}$ consists of a single point.

On the other hand, since $\ol{\map}^{-1}(t)\cap(\{0\}\times S^1\times\{x\})$ is a single point, this point has to be equal to $\Theta_t(0,x)$. Therefore,
$\Theta_t(0,x)$ is  onto $\ol{\map}^{-1}(t)\cap\{r=0\}$ so $\Theta_t(r,x)$ is onto $\ol{\map}^{-1}(t)$.
\end{proof}

As a corollary we will show the following result.
\begin{theorem}[Continuous Fibration Theorem]\label{thm:islocallytrivial}
The map $\ol{\map}\colon \ol{X}\to S^1$ is a continuous, locally trivial fibration.
\end{theorem}
\begin{proof}
We show that for any closed interval $I\subset S^1$, the preimage $\ol{X}_I:=\ol{\map}^{-1}(I)$ is homeomorphic to the product
$Y_I\colon = I\times [0,\varepsilon_0]\times M$ by a homeomorphism that preserves the fibers. Consider the map $\ol{\Theta}_I\colon Y_I\to X_I$
given by $\ol{\Theta}_I(t,x)=\ol{\Theta}_t(x)$ for $x\in[0,\varepsilon_0]\times M$. By Theorems~\ref{thm:graph} and \ref{thm:graph2}, this map is a bijection.
Moreover, its inverse is $\ol{\map}$, which is continuous by the Continuous Extension Lemma~\ref{prop:extend1}. A continuous bijection between compact sets
is a homeomorphism.  It is clear that $\ol{\Theta}_I$ preserves the fibers.
\end{proof}

\section{Constructing Seifert hypersurfaces based on $\map$}\label{sec:seifert}

\begin{theorem}
Let $t\in S^1$ be a non-critical value of the map $\map$ and $t\neq 0$. Then the closure of $\map^{-1}(t)$
 is a Seifert hypersurface for $M$ which is smooth up to boundary. Moreover, the $(n+1)$-dimensional volume
of $\map^{-1}(t)$ is finite.
\end{theorem}
\begin{proof}
Let $\Sigma=\map^{-1}(t)$.
By the implicit function theorem $\Sigma$ is a smooth open submanifold of $\R^{n+2}\setminus M$.
By Theorem~\ref{thm:mapisproper} we infer that $\Sigma$ is contained in some ball $B(0,R)$ for large $R$. This implies that $\Sigma\setminus N$ is compact.

The main problem is to show that boundary of the closure of $\Sigma$ is $M$. To this end we study the intersection $\Sigma_0:=\Sigma\cap (N\setminus M)$. Notice
that we have a diffeomorphism $\Sigma_0\cong \map^{-1}(t)\cap X$ via
the map $X\stackrel{\simeq}{\to} (N\setminus M)$.

Now $\Sigma_0$ is a smooth surface diffeomorphic to $(0,\varepsilon_0]\times M$. By Theorem~\ref{thm:graph}
the closure $\ol{\Sigma}_0$ of $\Sigma_0$ in $\ol{X}$ is homeomorphic to the product $[0,\varepsilon_0]\times M$. Under the map $\ol X\to N$
the closure $\ol{\Sigma}_0$ is mapped to the closure of $\Sigma$ in $N$.
It follows that the boundary of the closure of $\Sigma\cap N$ is $M$ itself.

To show the finiteness of the volume of $\Sigma$, notice that the area of $\Sigma\setminus N$ is finite, because $\Sigma\setminus N$ is smooth and
compact. The finiteness of the volume of $\Sigma\cap N$ follows from the Bounded Volume Lemma~\ref{prop:boundedarea}.
\end{proof}

In numerical applications calculating the map $\map$ in $N$ can be challenging due to the lack of a good bound for derivatives of $\map$ in $N$.
Therefore the following corollary should be useful.

\begin{proposition}
Choose $t\in S^1$, $t\neq 0$ to be a non-critical value of $\map$. Define $M'=\map^{-1}(t)\cap \partial N$. Let $\Sigma'=\map^{-1}(t)\setminus N$.
Then $M'$ is diffeomorphic to $M$, isotopic to $M$ as knots in $S^{n+2}$ and $\Sigma'$ is a smooth surface for $M'$.
\end{proposition}
\begin{proof}
The fact that $M'$ is diffeomorphic to $M$ follows from the Fibration Lemma~\ref{cor:locallytrivial}. The isotopy is given by
$M_r=\pi\circ \ol{\Phi}_t(\{r\}\times M)$, where $\ol{\Phi}$ is as in Theorem~\ref{thm:graph} and $\pi\colon X\to N$ is the projection.
By definition $\partial\Sigma'=M'$
and as $\Sigma'$ is closed and bounded it is also compact.
\end{proof}

\end{document}